\documentclass{amsart}

\usepackage{amsmath}
\usepackage{amsthm}
\usepackage{amsopn}
\usepackage{amssymb}
\usepackage[all]{xy}

\allowdisplaybreaks[1]
\newcommand{\ul}[1]{\underline{#1}}
\newcommand{\mc}[1]{\mathcal{#1}}
\newcommand{\mb}[1]{\mathbb{#1}}
\newcommand{\mr}[1]{\mathrm{#1}}

\newcommand{\mit}[1]{\mathit{#1}}
\newcommand{\mf}[1]{\mathfrak{#1}}
\newcommand{\abs}[1]{\left\lvert #1 \right\rvert}

\newcommand{\td}[1]{\widetilde{#1}}

\newcommand{\ZZ}{\mathbb{Z}}

\newcommand{\QQ}{\mathbb{Q}}

\newcommand{\FF}{\mathbb{F}}

\newcommand{\Alg}{\mathrm{Alg}}
\newcommand{\Mod}{\mathrm{Mod}}
\newcommand{\Id}{\mathrm{Id}}
\newcommand{\Sp}{\mathrm{Sp}}
\newcommand{\Top}{\mathrm{Top}}
\newcommand{\Coalg}{\mathrm{Coalg}}
\newcommand{\Lie}{\mathrm{Lie}}
\newcommand{\Comm}{\mathrm{Comm}}

\newcommand{\op}{\mathrm{op}}

 \newtheorem{thm}[equation]{Theorem}
 \newtheorem{cor}[equation]{Corollary}
 \newtheorem{lem}[equation]{Lemma}
 \newtheorem{prop}[equation]{Proposition}
 \newtheorem{ques}[equation]{Question}

\theoremstyle{definition}

 \newtheorem{ex}[equation]{Example}
 
 \newtheorem{rmk}[equation]{Remark}
\newtheorem*{thm*}{Theorem}
\newtheorem*{cor*}{Corollary}
\newtheorem*{lem*}{Lemma}
\newtheorem*{prop*}{Proposition}
\newtheorem*{defn*}{Definition}
\newtheorem*{ex*}{Example}
\newtheorem*{exs*}{Examples}
\newtheorem*{rmk*}{Remark}
\newtheorem*{claim*}{Claim}

\numberwithin{equation}{section}
\numberwithin{figure}{section}

\DeclareMathOperator{\Tor}{Tor}

\DeclareMathOperator{\Prim}{Prim}

\DeclareMathOperator{\TAQ}{TAQ}

\DeclareMathOperator{\Ho}{Ho}

\DeclareMathOperator*{\holim}{holim}
\DeclareMathOperator*{\hocolim}{hocolim}

\DeclareMathOperator*{\Tot}{Tot}

\title[Spectral algebra models]{Spectral algebra models of unstable $v_n$-periodic homotopy theory}
\author{Mark Behrens and Charles Rezk}
\date{\today}

\begin{document}

\begin{abstract}
We give a survey of a generalization of Quillen-Sullivan rational homotopy theory which gives spectral algebra models of unstable $v_n$-periodic homotopy types.  In addition to describing and contextualizing our original approach, we sketch two other recent approaches which are of a more conceptual nature, due to Arone-Ching and Heuts.  In the process, we also survey many relevant concepts which arise in the study of spectral algebra over operads, including topological Andr\'e-Quillen cohomology, Koszul duality, and Goodwillie calculus.
\end{abstract}

\maketitle

\tableofcontents

\section{Introduction}

In his seminal paper \cite{Quillen}, Quillen showed that there are equivalences of homotopy categories
$$ \Ho(\Top^{\ge 2}_\QQ) \simeq  \Ho(\mr{DG}\Coalg^{\ge 2}_\QQ) \simeq \Ho(\mr{DG}\Lie_\QQ^{\ge 1}) $$
between simply connected rational spaces, simply connected rational differential graded commutative coalgebras, and connected rational differential graded Lie algebras.  In particular, given a simply connected space $X$, there are \emph{models} of its rational homotopy type
\begin{align*}
C_\QQ(X) & \in \mr{DG}\Coalg_\QQ, \\
L_\QQ(X) & \in \mr{DG}\Lie_\QQ
\end{align*}
such that 
\begin{align*}
H_*(C_\QQ(X)) & \cong H_*(X;\QQ) \quad & \text{(isomorphism of coalgebras)}, \\
H_*(L_\QQ(X)) & \cong \pi_{*+1}(X) \otimes \QQ \quad & \text{(isomorphism of Lie algebras)}. 
\end{align*}
In the case where the space $X$ is of finite type, one can also extract its rational homotopy type from the dual $C_\QQ(X)^\vee$, regarded as a differential graded commutative algebra.  This was the perspective of Sullivan \cite{Sullivan}, whose notion of minimal models enhanced the computability of the theory.

The purpose of this paper is to give a survey of an emerging generalization of this theory where unstable rational homotopy is replaced by $v_n$-periodic homotopy.  

Namely, the Bousfield-Kuhn functor $\Phi_{K(n)}$ is a functor from spaces to spectra, such that the homotopy groups of $\Phi_{K(n)}(X)$ are a version of the unstable $v_n$-periodic homotopy groups of $X$.  We say that a space $X$ is $\Phi_{K(n)}$-\emph{good} if the Goodwillie tower of $\Phi_{K(n)}$ converges at $X$.  A theorem of Arone-Mahowald \cite{AroneMahowald} proves spheres are $\Phi_{K(n)}$-good.  

The main result is the following theorem (Theorem~\ref{thm:main}, Corollary~\ref{cor:main}).

\begin{thm}
There is a natural transformation (the ``comparison map'')
$$ c_X^{K(n)}: \Phi_{K(n)}(X) \rightarrow \TAQ_{S_{K(n)}}(S^X_{K(n)}) $$
which is an equivalence on finite $\Phi_{K(n)}$-good spaces.
\end{thm}

Here the target of the comparison map is the topological Andr\'e-Quillen cohomology of the $K(n)$-local Spanier-Whitehead dual of $X$ (regarded as a non-unital commutative algebra over the $K(n)$-local sphere), where $K(n)$ is the $n$th Morava $K$-theory spectrum.  We regard $S^X_{K(n)}$ as a commutative algebra model of the unstable $v_n$-periodic homotopy type of $X$, and the theorem is giving a means of extracting the unstable $v_n$-periodic homotopy groups of $X$ from its commutative algebra model.  A result of Ching \cite{Ching} implies that the target of the comparison map is an algebra over a spectral analog of the Lie operad.  As such, we regard the target as a Lie algebra model for the unstable $v_n$-periodic homotopy type of $X$. 

The original results date back to 2012, and are described in a preprint of the authors \cite{BKTAQ} which has (still?) not been published. The paper is very technical, and the delay in publication is due in part to difficulties in getting these technical details correct.
In the mean-time, Arone-Ching \cite{AroneChingvn} and Heuts \cite{Heuts2} have announced proofs which reproduce and expand on the authors' results using more conceptual techniques.  

The idea of this survey is to provide a means to disseminate the authors' original work until the original account is published.  As \cite{BKTAQ} is more of a forced march than a reflective ramble, it also seemed desirable to have a discussion which explained the main ideas without getting bogged down in the inevitable details one must contend with (which involve careful work with the Morava $E$-theory Dyer-Lashof algebra, amongst other things).  The approach of Arone-Ching uses a localized analog of their classification theory for Taylor towers, together with Ching's Koszul duality for modules over an operad.  Heuts' approach is a byproduct of his theory of polynomial approximations of $\infty$-categories.  Both of these alternatives, as we mentioned before, are more conceptual than our computational approach, but require great care to make precise.

This survey, by contrast, is written to convey the \emph{ideas} behind all three approaches, without delving into many details.  We also attempt to connect the theory with many old and new developments in spectral algebra.  We hope that the interested reader will consult cited sources for more careful treatments of the subjects herein.  In particular, all constructions are implicitly derived/homotopy invariant, and we invite the reader to cast them in his/her favorite model category or $\infty$-category.

\subsection*{Organization of the paper} $\quad$

{\it Section~\ref{sec:general}:} We describe the general notion of stabilization of a homotopy theory, and the Hess/Lurie theory of homotopy descent as a way of encoding unstable homotopy theory as ``stable homotopy theory with descent data''.

{\it Section~\ref{sec:Koszul}:} The equivalence between rational differential graded Lie algebras and rational differential graded commutative coalgebras is an instance of Koszul duality.  We describe the theory of Koszul duality, which provides a correspondence between algebras over an operad, and coalgebras over its Koszul dual.

{\it Section~\ref{sec:models}:} We revisit rational homotopy theory and recast it in spectral terms.  We also describe Mandell's work, which gives commutative algebra models of $p$-adic homotopy types.

{\it Section~\ref{sec:vn}:} We give an overview of chromatic ($v_n$-periodic) homotopy theory, both stable and unstable, and review the Bousfield-Kuhn functor.

{\it Section~\ref{sec:comparison}:} We define the comparison map, and state the main theorem in the case where $X$ is a sphere.

{\it Section~\ref{sec:proof}:} We give an overview of the proof of the main theorem in the case where $X$ is a sphere.  The proof involves Goodwillie calculus and the Morava $E$-theory Dyer-Lashof algebra, both of which we review in this section.

{\it Section~\ref{sec:consequences}:}  We explain how the main theorem extends to all finite $\Phi_{K(n)}$-good spaces.  We also discuss computational consequences of the theorem, most notably the work of Wang and Zhu.

{\it Section~\ref{sec:AroneChing}:}  After summarizing Ching's Koszul duality for modules over an operad, we give an exposition of the Arone-Ching theory of fake Taylor towers, and their classification of polynomial functors.  We then explain how they use this theory, in the localized context, to give a different proof (and strengthening) of the main theorem. 

{\it Section~\ref{sec:Heuts}:}  We summarize Heuts' theory of polynomial approximations of $\infty$-categories, and his general theory of coalgebra models of homotopy types.  We discuss Heuts' application of his general theory to Koszul duality, and to unstable $v_n$-periodic homotopy, where his theory also reproves and strengthens the main theorem.

\subsection*{Conventions}
\begin{itemize}
\item For a commutative ($E_\infty$) ring spectrum $R$, we shall let $\Mod_R$ denote the category of $R$-module spectra, with symmetric monoidal structure given by $\wedge_R$.  For $X, Y$ in $\Mod_R$, we will let $F_R(X,Y)$ denote the spectrum of $R$-module maps from $X$ to $Y$, and $X^\vee := F_R(X,R)$ denotes the $R$-linear dual. For a pointed space $X$, We shall let $R^X$ denote the function spectrum $F(\Sigma^\infty X, R)$. \vspace{7pt}

\item For $X$ a space or spectrum, we shall use $X^\wedge_p$ to denote its $p$-completion with respect to a prime $p$, $X_E$ to denote its Bousfield localization with respect to a spectrum $E$, and $X^{\ge n}$ to denote its $(n-1)$-connected cover.
\vspace{7pt}

\item For all but the last section, our homotopical framework will always implicitly take place in the context of relative categories: a category $\mc{C}$ with a subcategory $\mc{W}$ of ``equivalences'' \cite{DHKS} (in the last section we work in the context of $\infty$-categories).  The homotopy category  will be denoted $\Ho(\mc{C})$, and refers to the localization $\mc{C}[\mc{W}^{-1}]$. Functors between homotopy categories are always implicitly derived.  We shall use $\mc{C}(X,Y)$ to refer to the maps in $\mc{C}$, and $[X,Y]_{\mc{C}}$ to denote the maps in $\Ho(\mc{C})$.  We shall use $\ul{\mc{C}}(X,Y)$ to denote the derived mapping space.  
\vspace{7pt}

\item $\Top_*$ denotes the category of pointed spaces (with equivalences the weak homotopy equivalences), $\Sp$ the category of spectra (with equivalences the stable equivalences), and for a spectrum $E$, $(\Top_*)_E$ and $\Sp_E$ denote the variants where we take the equivalences to be the $E$-homology isomorphisms. 
\vspace{7pt}

\item All operads $\mc{O}$ in $\Mod_R$ are assumed to be reduced, in the sense that $\mc{O}_0 = \ast$ and $\mc{O}_1 = R$.  We shall let $\Alg_\mc{O}$ denote the category of $\mc{O}$-algebras.  As spelled out in greater detail in Section~\ref{sec:Koszul}, $\TAQ^\mc{O}$ will denote topological Andr\'e-Quillen homology, and $\TAQ_\mc{O}$ will denote topological Andr\'e-Quillen cohomology (its $R$-linear dual).  In the case where $\mc{O} = \Comm_R$, the (reduced) commutative operad in $\Mod_R$, we shall 
let $\TAQ^R$ (respectively $\TAQ_R$) denote the associated topological Andr\'e-Quillen homology (respectively cohomology).\footnote{This is slightly non-standard, as $\Comm$-algebras are the same thing as \emph{non-unital} commutative algebras in $\Mod_R$.  However, as we explain in Section~\ref{sec:Koszul}, the category of such is equivalent to the category of augmented commutative $R$-algebras.} 
\end{itemize}

\subsection*{Acknowledgments}
The authors benefited greatly from conversations with Greg Arone, Michael Ching, Bill Dwyer, Rosona Eldred, Sam Evans, John Francis, John Harper, Gijs Heuts, Mike Hopkins, Nick Kuhn, Jacob Lurie, Mike Mandell, Akhil Mathew, Anibal Medina, Lennart Meier, Luis Alexandre Pereira, and Yifei Zhu.  The authors are grateful to Norihiko Minami for encouraging this submission to these conference proceedings, honoring the memory of Tetsusuke Ohkawa. The authors would also like to thank the referee for his/her many useful comments and corrections.  Both authors were supported by grants from the NSF.

\section{Models of ``unstable homotopy theory''}\label{sec:general}

The approach to unstable homotopy theory we are considering fits into a general context, which we will now describe.

\subsection*{Stable homotopy theories}

As Quillen points out in \cite{HA}, any pointed model category $\mc{C}$ comes equipped with a notion of suspension $\Sigma_{\mc{C}}$ and loops $\Omega_\mc{C}$, given by
\begin{gather*}
\Sigma_\mc{C} X = \hocolim(* \leftarrow X \rightarrow *), \\
\Omega_\mc{C} X = \holim(* \rightarrow X \leftarrow *).
\end{gather*} 
This gives the notion of a category $\Sp(\mc{C})$ of spectra in $\mc{C}$.  With hypotheses on $\mc{C}$, and a suitable notion of stable equivalence (see, for example, \cite{Schwede}, \cite{Hovey}), $\Sp(\mc{C})$ is a model for the \emph{stabilization} of $\mc{C}$ (in the sense of \cite{Lurie}).
There are adjoint functors
\begin{equation}\label{eq:sigmainftyomegainfty}
 \Sigma^\infty_\mc{C}: \Ho(\mc{C}) \leftrightarrows \Ho(\Sp(\mc{C})): \Omega^\infty_{\mc{C}}.
\end{equation}
We regard $\Ho(\mc{C})$ as the unstable homotopy theory of $\mc{C}$, and $\Ho(\Sp(\mc{C}))$ as the stable homotopy theory of $\mc{C}$.  

\subsection*{The fundamental question}

Typically, the unstable homotopy theory is \emph{more complicated} than the stable homotopy theory.  One would therefore like to think that an unstable homotopy type is a stable homotopy type with extra structure.  More specifically:
\begin{ques}\label{ques:fund}
Is there an algebraic structure ``?'' on $\Sp(\mc{C})$ and functors:
$$ \mf{A}: \Ho(\mc{C}) \leftrightarrows \Ho(\Alg_?(\Sp(\mc{C}))) : \mf{E} $$
so that $X \simeq \mf{E} \mf{A} (X)$ (natural isomorphism in the homotopy category)? 
\end{ques}
If so, we say that $?$-algebras model the unstable homotopy types of $\mc{C}$.

\begin{rmk}$\quad$
\begin{enumerate}
\item Often, one must restrict attention to certain subcategories of $\Ho(\mc{C})$, $\Ho(\Alg_?)$ to get something like this (e.g. $1$-connected rational unstable homotopy types).

\item One can hope for more: is $\mf{A}$ fully faithful?  Can we then
  characterize the essential image?

\item When $(\mf{A},\mf{E})$ form  an adjoint pair, we can say
  something sharper: in this case, there is always a \emph{canonical}
  equivalence between the full subcategories
\[
\Ho\bigl\{\text{$X\in \mc{C}$ s.t.\ $X\simeq \mf{E}\mf{A}(X)
  $}\bigr\} \simeq 
\Ho\bigl\{\text{$A\in
  \Alg_{?}(\Sp(\mc{C}))$ s.t.\ 
    $A\simeq \mf{A}\mf{E}(A)$}\bigr\}.
\]
This identifies both the ``good'' subcategory of $\Ho(\mc{C})$ and its
essential image under $\mf{A}$, and shows that $\mf{A}$ is fully
faithful on this subcategory.
\end{enumerate}  
\end{rmk}

\begin{ex}
In the case of $\mc{C} = (\Top_*)_\QQ$ --- rational pointed spaces --- the stabilization is rational spectra $\Sp_\QQ$.  We have
$$ \Ho(\Sp_\QQ) \simeq \Ho(\mr{Ch}_\QQ), $$
where $\mr{Ch}_\QQ$ denotes rational $\ZZ$-graded chain complexes.  In this context Quillen's work provides two answers to Question~\ref{ques:fund}: the algebraic structure can be taken to be either commutative coalgebras or Lie algebras.
\end{ex}

\subsection*{Homotopy descent}

The theory of homotopy decent of Hess \cite{Hess} and Lurie \cite{Lurie} 
(see also \cite{AroneChing})
provides a canonical candidate answer to Question~\ref{ques:fund}.   Namely, the adjunction~(\ref{eq:sigmainftyomegainfty}) gives rise to a comonad $\Sigma^\infty_{\mc{C}} \Omega^\infty_{\mc{C}}$ on $\Sp(\mc{C})$, and for any $X \in \mc{C}$, the spectrum $ \Sigma^\infty_{\mc{C}} X$ is a coalgebra for this comonad.\footnote{One should regard this coalgebra structure as ``descent data''.}  Thus one can regard the functor $\Sigma^\infty_\mc{C}$ as refining to a functor
$$ \mf{A}: \Ho(\mc{C}) \rightarrow \Ho(\Coalg_{\Sigma^\infty_\mc{C}\Omega^\infty_\mc{C}}). $$
Asking for this to be an equivalence is asking for the adjunction to
be ``comonadic''.  It is typically only reasonable to expect that one
gets an equivalence between suitable subcategories of these two
categories.  
Even then, this may be of little use if there is no
explicit understanding of what it means to be a
$\Sigma^\infty_{\mc{C}} \Omega^\infty_{\mc{C}}$-coalgebra. 

\begin{ex}
Suppose that $\mc{C} = \Top_*$, the category of pointed spaces.  Then there is always a map
\begin{equation}\label{eq:Qcomp}
 X \rightarrow C(\Omega^\infty, \Sigma^\infty \Omega^\infty, \Sigma^\infty X)
 \end{equation}
where $C(-,-,-)$ denote the comonadic cobar construction.  Explicitly, 
$$ C(\Omega^\infty, \Sigma^\infty \Omega^\infty, \Sigma^\infty X) = \Tot (QX \Rightarrow QQX \Rrightarrow \cdots ), $$
the Bousfield-Kan $Q$-completion of $X$.  It follows that the map
(\ref{eq:Qcomp}) is an equivalence for $X$ nilpotent, and for
nilpotent spaces the unstable homotopy type can be recovered from the
$\Sigma^\infty \Omega^\infty$-comonad structure on $\Sigma^\infty X$.
But what does it mean explicitly to endow a spectrum with a
$\Sigma^\infty \Omega^\infty$-coalgebra structure?  This seems to be a
difficult question, but Arone, Klein, Heuts, and others have partial
information (see \cite{Klein}, \cite{Heuts1}). Rationally, however,
$\Sigma^\infty \Omega^\infty$ is equivalent (on connected spaces) to
the free commutative 
coalgebra functor, and coalgebras for this comonad are therefore
rationally equivalent to commutative coalgebras.   
\end{ex}

\section{Koszul duality}\label{sec:Koszul}

The equivalence
$$ \Ho(\mr{DG}\Coalg^{\ge 2}_\QQ) \simeq \Ho(\mr{DG}\Lie_\QQ^{\ge 1}) $$
mentioned in the introduction is an instance of \emph{Koszul duality} \cite{GinzburgKapranov}, \cite{GetzlerJones}, \cite{FrancisGaitsgory}, \cite{Fresse}, \cite{AyalaFrancis}, \cite{Lurie}, \cite{ChingHarper}.
In this section we will attempt to summarize the current state of affairs to the best of our abilities.

Let $R$ be a commutative ring spectrum, and let $\mc{O}$ be an operad in $\Mod_R$.  \emph{All operads $\mc{O}$ in this paper are assumed to be {\bf reduced}}: $\mc{O}_0 = *$ and $\mc{O}_1 = R$.

We shall let $\Alg_\mc{O} = \Alg_{\mc{O}}(\Mod_R)$ denote the category of $\mc{O}$-algebras.  An equivalence of $\mc{O}$-algebras is a map of $\mc{O}$-algebras whose underlying map of spectra is an equivalence.\footnote{We refer the reader to \cite{HarperHess} for a thorough treatment of the homotopy theory of $\mc{O}$-algebras suitable for our level of generality.  We advise the reader that some of the technical details in this reference are correctly dealt with in \cite{PereiraOperad}, \cite{KuhnPereira}.}  Note that since the operad $\mc{O}$ is reduced, the category $\Alg_\mc{O}$ is pointed, with $*$ serving as both the initial and terminal object.
There is a free-forgetful adjunction
$$ \mc{F}_\mc{O}: \Mod_R \leftrightarrows \Alg_\mc{O}: \mc{U} $$
where
\begin{equation}\label{eq:FO}
\mc{F}_\mc{O}(X) = \bigvee_i \left( \mc{O}_i \wedge_{R} X^{\wedge_R i} \right)_{\Sigma_i}
\end{equation}
is the free $\mc{O}$-algebra generated by $X$.  We shall abusively also use $\mc{F}_{\mc{O}}$ to denote the associated monad on $\Mod_R$, so that
$\mc{O}$-algebras are the same thing as $\mc{F}_\mc{O}$-algebras:
$$ \Alg_\mc{O} \simeq \Alg_{\mc{F}_\mc{O}}. $$

\subsection*{Topological Andr\'e-Quillen homology}

Because $\mc{O}$ is reduced, there is a natural transformation of monads
$$ \epsilon: \mc{F}_{\mc{O}} \rightarrow \Id. $$
For $A$ an $\mc{O}$-algebra, its module of \emph{indecomposables} $QA$ 
is defined to be the coequalizer of $\epsilon$ and the $\mc{F}_{\mc{O}}$-algebra structure map:
$$ \mc{F}_{\mc{O}}(A) \rightrightarrows A \rightarrow QA. $$
The functor $Q$ has a right adjoint
$$ Q : \Alg_\mc{O} \leftrightarrows \Mod_R : \mr{triv} $$
where, for an $R$-module $X$, the $\mc{O}$-algebra $\mr{triv} X$ is given by endowing $X$ with $\mc{O}$-algebra structure maps:
\begin{align*}
\mc{O}_1 \wedge_R X & = R \wedge_R X \xrightarrow{\approx} X, \\
\mc{O}_n \wedge_R X^{n} & \xrightarrow{\ast} X, \qquad n \ne 1. 
\end{align*}
The \emph{topological Andr\'e-Quillen homology} of $A$ is defined to be the left derived functor
$$ \TAQ^{\mc{O}}(A) := \mb{L}Q A. $$
It is effectively computed as the realization of the monadic bar construction:
$$ \TAQ^\mc{O}(A) \simeq B(\Id, \mc{F}_\mc{O}, A). $$
The Topological Andr\'e-Quillen cohomology is defined to be the $R$-linear dual of $\TAQ^{\mc{O}}$:
$$ \TAQ_\mc{O}(A) := \TAQ^\mc{O}(A)^\vee. $$

Suppose $R= Hk$ is the Eilenberg-MacLane spectrum associated to a $\QQ$-algebra $k$, $\mc{O}$ is the commutative operad (see Example~\ref{ex:comm} below), and $A$ is the Eilenberg-MacLane $\mc{O}$-algebra associated to an ordinary augmented commutative $k$-algebra. Then we can regard $\TAQ^\mc{O}$ as being an object of the derived category of $k$ under the equivalence
$$ \Ho(\Mod_{Hk}) \simeq \Ho(\mr{Ch}_k) $$
and we recover classical Andr\'e-Quillen homology.  Basterra defined $\TAQ$ for commutative $R$-algebras for arbitrary commutative ring spectra $R$, and showed that the monadic bar construction gives a formula for it \cite{Basterra}.  The case of general topological operads was introduced in \cite{BasterraMandell}; this work was extended to the setting of spectral operads in \cite{Harper} (see also  \cite{GoerssHopkins}).  

The important properties of $\TAQ^{\mc{O}}$ are:
\begin{enumerate}
\item $\TAQ^{\mc{O}}$ is excisive --- it takes homotopy pushouts of $\mc{O}$-algebras to homotopy pullbacks of $R$-modules (which are the same as homotopy pushouts in this case),
\item $\TAQ^{\mc{O}}(\mc{F}_\mc{O}(X)) \simeq X$ --- this is a consequence of the fact that $Q\mc{F}_{\mc{O}}X \approx X$.
\end{enumerate}
(1) and (2) above imply that if $A$ is built out of free $\mc{O}$-algebra cells, $\TAQ^{\mc{O}}(A)$ is built out of $R$-module cells in the same dimensions.  In this way, $\TAQ$ provides information on the ``cell structure'' of an $\mc{O}$-algebra.

\begin{ex}\label{ex:comm}
The (reduced) commutative operad $\mr{Comm} = \mr{Comm}_R$ is given by
$$ \mr{Comm}_i = 
\begin{cases}
*, & i = 0, \\
R, & i \ge 1.
\end{cases} $$
A $\mr{Comm}_R$-algebra is a \emph{non-unital} commutative $R$-algebra.  The category of non-unital commutative $R$-algebras is equivalent to the category of augmented commutative $R$-algebras:
$$ \Alg_{\mr{Comm}_R} \simeq (\Alg_R)_{/R}. $$
Given an augmented commutative $R$-algebra $A$, the augmentation ideal $IA$ given by the fiber
$$ IA \rightarrow A \xrightarrow{\epsilon} R $$
is the associated non-unital commutative algebra.  In this setting, we have
$$ \TAQ^{\mr{Comm}_R}(IA) \simeq \TAQ^R(A) $$
where $\TAQ^R(-)$ is the $\TAQ$ of \cite{Basterra}.
\end{ex}

\subsection*{The stable homotopy theory of $\mc{O}$-algebras}

The following theorem was first proven in the context of simplicial commutative rings in \cite{Schwede}, in the context of $R$ arbitrary and $\mc{O} = \mr{Comm}$ in \cite{BasterraMcCarthy} and \cite{BasterraMandell}, and $R$ and $\mc{O}$ arbitrary in \cite{Pereira1} (see also \cite{FrancisGaitsgory}, \cite[Thm.~7.3.4.13]{Lurie}).

\begin{thm}
There is an equivalence of categories
$$ \Ho(\Sp(\Alg_\mc{O})) \simeq \Ho(\Mod_R). $$
Under this equivalence, the functors
$$ \Sigma^\infty_{\Alg_\mc{O}}
: \Ho(\Alg_\mc{O}) \leftrightarrows \Ho(\Mod_R): \Omega^\infty_{\Alg_\mc{O}} $$
are given by
\begin{align*}
\Sigma^\infty_{\Alg_\mc{O}}A & \simeq \TAQ^\mc{O}(A), \\
\Omega^\infty_{\Alg_\mc{O}} X & \simeq \mr{triv} X. 
\end{align*}
\end{thm}

The adjunction above extends to derived mapping spaces, and gives the following (compare with \cite{Basterra}).

\begin{cor}\label{cor:TAQ}
The spaces of the $\TAQ_{\mc{O}}$-spectrum are given by
$$ \Omega^\infty \Sigma^{n} \TAQ_\mc{O}(A) \simeq \ul{\Alg}_{\mc{O}} (A, \mr{triv} \Sigma^n R). $$
\end{cor}

\begin{proof}
We have
\begin{align*}
\ul{\Alg}_\mc{O}(A, \mr{triv} \Sigma^{n} R) & \simeq \ul{\Mod}_R(\TAQ^\mc{O}(A), \Sigma^{n} R) \\
& \simeq \ul{\Mod}_R(R, \Sigma^n \TAQ_\mc{O}(A)) \\
& \simeq \Omega^\infty \Sigma^n \TAQ_\mc{O}(A).
\end{align*}
\end{proof}

\subsection*{Divided power coalgebras}

Question~\ref{ques:fund} clearly has a tautological answer when
$\mc{C}=\Alg_{\mc{O}}$: it  consists of $\mc{O}$-algebras in
$\Sp(\mc{C})\simeq\Mod_R$. 
However, this is \emph{not} the canonical spectral algebra model given
by the theory of homotopy descent of Section~\ref{sec:general} --- we
should be considering the $\Sigma^\infty_{\Alg_\mc{O}} 
\Omega^\infty_{\Alg_\mc{O}}$-coalgebra $\TAQ^{\mc{O}}(A)$ as a candidate spectral algebra model for $A$.

But what does it mean to be a $\Sigma^\infty_{\Alg_\mc{O}} \Omega^\infty_{\Alg_\mc{O}}$-coalgebra?  The answer, according to \cite{FrancisGaitsgory} and \cite{ChingHarper}, is a \emph{divided power coalgebra over the Koszul dual $B\mc{O}$}.  Let us unpack what this means.  

For any symmetric sequence $\mc{Y} = \{ \mc{Y}_i \}$ of $R$-modules, one can use (\ref{eq:FO}) to define a functor 
$$ \mc{F}_\mc{Y}: \Mod_R \rightarrow \Mod_R. $$
The category of symmetric sequences of $R$-modules possesses a monoidal structure $\circ$ called the composition product, such that 
$$ \mc{F}_{\mc{Y}} \circ \mc{F}_\mc{Z} = \mc{F}_{\mc{Y} \circ \mc{Z}}. $$
The monoids associated to the composition product are precisely the
operads in $\Mod_R$.  The unit for this monoidal structure is the
symmetric sequence $1_R$ with  
$$ (1_R)_i = 
\begin{cases}
R, & i = 1, \\
*, & i \neq 1.
\end{cases}
$$
Every reduced operad $\mc{O}$ in $\Mod_R$ is augmented over $1_R$.  The \emph{Koszul dual} of $\mc{O}$ is the symmetric sequence obtained by forming the bar construction with respect to the composition product
$$ B\mc{O} := B(1_R, \mc{O}, 1_R) = |1 \Leftarrow \mc{O} \Lleftarrow \mc{O} \circ \mc{O} \cdots |. $$
Ching showed that $B\mc{O}$ admits a cooperad structure \cite{Ching}. 

\begin{ex}\label{ex:Lie}
Suppose $R = H\QQ$, so we can replace $\Mod_R$ with $\mr{Ch}_{\QQ}$.  Take $\mc{O} = \mr{Lie}_\QQ$, the Lie operad.  Then we have $B\mr{Lie}_\QQ = s\mr{Comm}_\QQ^\vee$ the suspension of the commutative cooperad \cite{GinzburgKapranov}, \cite{Ching}.\footnote{In general, for a (co)operad $\mc{O}$, the \emph{suspension} of the (co)operad $s\mc{O}$ is a new (co)operad for which $(s\mc{O})_i \simeq \Sigma^{i-1} \mc{O}_i$ (nonequivaraintly), with the property that an $s\mc{O}$-(co)algebra structure on $X$ is the same thing as an $\mc{O}$-(co)algebra structure on $\Sigma X$ \cite{MarklSniderStasheff}, \cite{AroneKankaanrinta}.} 
\end{ex}

\begin{ex}\label{ex:spectrallie}
In the case of $R = S$, the sphere spectrum, and $\mc{O}$ the commutative operad, Ching showed that
$$ B\mr{Comm}_S \simeq (\partial_* \Id_{\Top_*})^\vee $$
the duals of the Goodwillie derivatives of the identity functor on $\Top_*$\footnote{This identification used the computation of $\partial_* \Id_{\Top_*}$ of \cite{Johnson} and \cite{AroneMahowald} as input.} \cite{Ching}.  He also showed that with respect to the resulting operad structure on $\partial_*\Id_{\Top_*}$, we have 
$$ sH_* \partial_* \Id_{\Top_*} \cong \Lie_\ZZ. $$
As such, we will \emph{define} the shifted spectral Lie operad as
$$ s^{-1} \Lie_S := \partial_*\Id_{\Top_*}. $$
\end{ex}

 Following \cite{ChingHarper}, we have for an $R$-module $X$:
\begin{align*}
\Sigma^\infty_{\Alg_\mc{O}} \Omega^\infty_{\Alg_\mc{O}} X & \simeq \TAQ^\mc{O}(\mr{triv} X) \\
& \simeq B(\Id, \mc{F}_\mc{O}, \mr{triv} X) \\
& \simeq \mc{F}_{B\mc{O}} X.
\end{align*}
If $R$ and $\mc{O}$ are connective, and $X$ is connected, we have 
$$
\mc{F}_{B\mc{O}} X \simeq \prod_i \left( B\mc{O}_i \wedge_{R} X^{\wedge_R i} \right)_{\Sigma_i}.
$$
Thus, at least on the level of the homotopy category, the data of a $\Sigma^\infty_{\Alg_\mc{O}} \Omega^\infty_{\Alg_\mc{O}}$-coalgebra $C$ corresponds to the existence of a collection of coaction maps:
$$ \psi_i: C \rightarrow \left( B\mc{O}_i \wedge_{R} C^{\wedge_R i} \right)_{\Sigma_i}. $$
The term \emph{divided power} comes from the fact that a standard
coalgebra over a cooperad consists of coaction maps into the $\Sigma_i$-fixed points rather than the $\Sigma_i$-orbits.  

The general notion of a divided power (co)algebra over a (co)operad goes back to Fresse (see\cite{Fressedp}, \cite{Fresse}).  
For a precise definition of divided power coalgebras in the present homotopy-coherent context, we refer the reader to \cite{FrancisGaitsgory}, \cite{Heuts1}.  In this language, we have functors
\begin{equation}\label{eq:Koszulduality}
\TAQ^{\mc{O}}: \Ho(\Alg_\mc{O}) \leftrightarrows \Ho(\mr{d.p.}\Coalg_{B\mc{O}}): \mf{E}.
\end{equation}

\subsection*{Instances of Koszul duality}

The following ``Koszul Duality'' theorem (a special case of a general conjecture of Francis-Gaitsgory \cite{FrancisGaitsgory}) generalizes Quillen's original theorem, as well as subsequent work in the algebraic context \cite{GinzburgKapranov}, \cite{GetzlerJones}, \cite{Fresse}, \cite{SchlessingerStasheff}.

\begin{thm}[Ching-Harper \cite{ChingHarper}]\label{thm:Koszulduality}
In the case where $R$ and $\mc{O}$ are connective, the functors (\ref{eq:Koszulduality}) restrict to give an equivalence of categories
$$ \Ho(\Alg^{\ge 1}_{\mc{O}}) \simeq \Ho(\mr{d.p.}\Coalg^{\ge 1}_{B\mc{O}}). $$
\end{thm}

\begin{ex}\label{ex:LieKoszul}
Returning to the context of $R = H\QQ$, and $\mc{O} = \mr{Lie}_\QQ$ of Example~\ref{ex:Lie}, Theorem~\ref{thm:Koszulduality} recovers Quillen's original theorem:
$$ \Ho(\Alg^{\ge 1}_{\mr{Lie}_\QQ}) \simeq \Ho(\Coalg^{\ge 1}_{s\mr{Comm}^\vee_\QQ}) \simeq \Ho(\Coalg^{\ge 2}_{\mr{Comm}^\vee_\QQ}). $$
Note that we have not mentioned divided powers.  This is because, rationally, coinvariants and invariants with respect to finite groups are isomorphic via the norm map, so \emph{every} rational coalgebra is a divided power coalgebra.
\end{ex}

\section{Models of rational and $p$-adic homotopy theory}\label{sec:models}

In this section we will return to Quillen-Sullivan theory, and a $p$-adic analog studied by Kriz, Goerss, Mandell, and Dwyer-Hopkins.

\subsection*{Rational homotopy theory, again}

We begin by recasting Quillen-Sullivan theory into the language of spectral algebra.  This in some sense defeats the original purpose of the theory --- which was to encode rational homotopy theory in an \emph{algebraic} category where you can literally write down the models in terms of generators, relations and differentials, but our recasting of the theory will motivate what follows.

Consider the functors
\begin{align*}
H\QQ \wedge -: \Ho((\Top_*)_\QQ) & \rightarrow \Ho(\Coalg_{\mr{Comm}_{H\QQ}^\vee}), \\
H\QQ^{-}: \Ho((\Top_*)_\QQ)^\op & \rightarrow \Ho(\Alg_{\mr{Comm}_{H\QQ}}).
\end{align*}
Essentially, for $X \in \Top_*$, $H\QQ \wedge X$ is a spectral model for the reduced chains on $X$, and $H\QQ^X$ is a spectral model for the reduced cochains of $X$.  The commutative coalgebra/algebra structures come from the diagonal
$$ \Delta: X \rightarrow X \wedge X. $$
The two functors are related by $H\QQ^X = (H\QQ \wedge X)^\vee$.
If $X$ is of finite type, there is no loss of information in using the cochains $H\QQ^X$.  There is a definite advantage to working with algebras rather than coalgebras if you like model categories.\footnote{For suitable monads $\mb{M}$ on cofibrantly generated model categories $\mc{C}$ it is typically straightforward to place induced model structures on $\Alg_{\mb{M}}$ \cite{Hirschhorn} --- coalgebras over comonads are more difficult to handle.  This may be an instance where there is a definite advantage in working with $\infty$-categories.  However, we also point out that Hess-Shipley \cite{HessShipley} give a useful framework which in practice can often give model category structures on categories of coalgebras over comonads.}

Quillen's theorem implies these functors restrict to give  equivalences of categories:
\begin{align*}
H\QQ \wedge (-): \Ho(\Top^{\ge 2}_\QQ) & \xrightarrow{\simeq} \Ho(\Coalg^{\ge 2}_{\mr{Comm}_{H\QQ}^\vee}), \\
H\QQ^{(-)}: \Ho(\Top^{\ge 2, \mr{f.t.}}_\QQ) & \xrightarrow{\simeq} \Ho(\Alg^{\le -2, \mr{f.t.}}_{\mr{Comm}_{H\QQ}}).
\end{align*}
His Lie algebra models then come from applying Koszul duality (see Example~\ref{ex:LieKoszul}).

\subsection*{$p$-adic homotopy theory}

Fix a prime $p$.
Analogous approaches to $p$-adic homotopy theory using cosimplicial commutative algebras, simplicial commutative coalgebras, $E_\infty$-algebras in chain complexes, and commutative algebras in spectra were developed respectively by Kriz \cite{Kriz}, Goerss \cite{Goerss}, Mandell \cite{Mandell}, and Dwyer-Hopkins (see \cite{Mandell}).  We will focus on the spectral algebra setting, which is closely tied to Mandell's algebraic setting.

The basic idea in these approaches is to replace the role of $H\QQ$ with the role of $H\bar\FF_p$.  Consider the cochain functor with $\bar\FF_p$-coefficients on $p$-complete spaces:
$$ H\bar\FF_p^{(-)}:  \Ho((\Top_*)_{\ZZ_p})^\op \rightarrow \Ho(\Alg_{\mr{Comm}_{H\bar\FF_p}}). $$

\begin{thm}[Mandell \cite{Mandell}]
The $\bar\FF_p$-cochains functor gives a fully faithful embedding
\begin{equation}\label{eq:Mandell}
H\bar\FF_p^{(-)}:  \Ho((\Top_*)^{\mr{nilp},\mr{f.t.}}_{\ZZ_p})^\op \hookrightarrow \Ho(\Alg_{\mr{Comm}_{H\bar\FF_p}}).
\end{equation}
of the homotopy category of nilpotent $p$-complete spaces of finite type into the homotopy category of commutative $H\bar\FF_p$-algebras.
\end{thm}

\begin{rmk}
Actually, the functor (\ref{eq:Mandell}) induces an equivalence on derived mapping spaces.
Mandell also computes the effective image of this functor.
\end{rmk}

\begin{rmk}\label{rem:goerss-coalgebras}
The approach of \cite{Goerss} suggests that the finite type hypothesis could be removed if one worked with $H\bar\FF_p$-coalgebras.
\end{rmk}

What goes wrong when using $H\FF_p$ instead of $H\bar\FF_p$?  Because the $\bar\FF_p$-cochains are actually defined over $\FF_p$, there is a continuous action of 
$$ \mr{Gal} := \mr{Gal}(\bar{\FF}_p/{\FF}_p) \cong \widehat{\ZZ} $$ 
on $H\bar\FF_p^X$, with homotopy fixed points:
$$ (H\bar\FF_p^X)^{h\mr{Gal}(\bar{\FF}_p/\FF_p)} \simeq H\FF_p^X. $$
It follows that for $X$, $Y$ nilpotent and of finite type, we have
\begin{align*}
\ul{\Alg}_{\mr{Comm}_{H\FF_p}}(H\FF_p^Y, H\FF_p^X) & \simeq 
\ul{\Alg}_{\mr{Comm}_{H\bar\FF_p}}(H\bar\FF_p^Y, H\bar\FF_p^X)^{h\mr{Gal}} \\
& \simeq \ul{\Top}_*(X^\wedge_p,Y^\wedge_p)^{h\mr{Gal}}.
\end{align*}
However, the action of $\mr{Gal}$ on $\ul{\Top}_*(X^{\wedge}_p,Y^{\wedge}_p)$ is trivial, so we have
\begin{align*}
\ul{\Top}_*(X^\wedge_p,Y^\wedge_p)^{h\mr{Gal}} 
& \simeq \ul{\Top}_*(X^\wedge_p,Y^\wedge_p)^{B\ZZ} \\
& \simeq L\ul{\Top}_*(X^\wedge_p,Y^\wedge_p) \qquad \text{(the free loop space).}\\
\end{align*}
In unpublished work (closely related to \cite{Mandellcochains}), Mandell has shown the same holds for $H\FF_p$ replaced by $S_p$, the $p$-adic sphere spectrum when $X$ and $Y$ are additionally assumed to be finite:
$$ \ul{\Alg}_{\mr{Comm}}(S_p^Y, S_p^X) \simeq L\ul{\Top}_*(X^\wedge_p,Y^\wedge_p). $$
In fact, Mandell has shown the integral cochains functors gives a faithful embedding of the integral homotopy category into the category of integral $E_\infty$-algebras \cite{Mandellcochains}
$$ \Ho(\Top_*^{\mr{nilp}, \mr{f.t.}})^{op} \hookrightarrow \Ho(\Alg_{E_\infty}(\mr{Ch}_{\ZZ})) $$
Medina has recently proven a related statement using $E_\infty$-coalgebras \cite{Medina}, and Blomquist-Harper have recently announced another setup using coalgebra structures on integral chains \cite{BlomquistHarper}.
In unpublished work, Mandell has a similar result for commutative $S$-algebras: the Spanier-Whitehead dual functor gives a faithful embedding:
$$ S^{(-)}: \Ho(\Top_*^{\mr{nilp}, \mr{finite}})^{op} \hookrightarrow \Ho(\Alg_{\Comm}(\Sp)). $$

\subsection*{Where are the $p$-adic Lie algebras?}

There is no known ``Lie algebra model'' for unstable $p$-adic homotopy
theory.  One of the problems is that, unlike the rational case,
commutative $H\bar\FF_p$-coalgebras do not automatically come equipped
with divided power structures, so Koszul duality does not seem to
apply (cf.\ the rational analogue of Example~\ref{ex:Lie}).  Applying
Koszul duality in the other direction, to get a ``divided power Lie 
coalgebra model'' (via a Koszul duality equivalence with commutative
algebras) is 
fruitless as well, since $\TAQ^{H\bar\FF_p}(H\bar\FF_p^X)\simeq \ast$
for any finite-type nilpotent $X$ (Thm.\ 3.4
of \cite{Mandellcochains}).  

One indication that one should not expect a Lie algebra model for $p$-adic homotopy types is that rationally, the composite
$$ \Ho(\Sp^{\ge 2}_\QQ) \xrightarrow{\Omega^\infty(-)} \Ho(\Top^{\ge 2}_\QQ) \xrightarrow[\simeq]{L_\QQ} \Ho(\Alg^{\ge 1}_{\mr{Lie}_\QQ}) $$
is given by 
$$ L_\QQ (\Omega^\infty Z) \simeq \mr{triv} \Sigma^{-1} Z. $$
where we give the spectrum $\Sigma^{-1} Z$ the \emph{trivial} Lie bracket.  This, strangely, means that a simply connected rational homotopy type is an infinite loop space if and only if its associated Lie algebra is equivalent to one with a trivial bracket.  There is thus a functor
\begin{equation}\label{eq:Phi0}
\Phi_0: \Ho(\Top^{\ge 2}_\QQ) \rightarrow \Ho(\Sp_\QQ)
\end{equation}
given by forgetting the Lie algebra structure on $L_\QQ$.  For a $1$-connected spectrum $Z$, we have
$$ \Phi_0 \Omega^{\infty} Z \simeq Z_\QQ, $$
i.e., we can recover the rationalization of the spectrum from its $0$th space.  It follows that rationally, simply connected infinite loop spaces have \emph{unique} deloopings.  An analogous fact does not hold for $p$-adic infinite loop spaces.

\section{$v_n$-periodic homotopy theory}\label{sec:vn}

In both the case of rational homotopy theory, and $p$-adic homotopy theory, there are notions of ``homotopy groups'' and ``homology groups''.  In the rational case, we have
\begin{align*}
\text{rational homotopy} & = \pi_*(X) \otimes \QQ, \\
\text{rational homology} & = H_*(X;\QQ).
\end{align*} 
The appropriate analogs in the $p$-adic case are 
\begin{align*}
\text{mod $p$ homotopy} & = \pi_*(X;M(p)) := [\Sigma^* M(p), X]_{\Top_*}, \\
\text{mod $p$ homology} & = H_*(X;\FF_p).
\end{align*} 
For $1$-connected spaces, a map is a rational homotopy isomorphism if and only if it is a rational homology isomorphism, and similarly, a map is a mod $p$ homotopy isomorphism if and only if it is a mod $p$ homology isomorphism.

The idea of chromatic homotopy theory is that a $p$-local homotopy type is built out of monochromatic (or $v_n$-periodic) layers, and that elements of $p$-local homotopy groups fit into periodic families of different frequencies.  The $v_n$-periodic homotopy groups isolate the elements in a particular frequency.  The associated homology theory is the $n$th Morava $K$-theory.

\subsection*{Stable $v_n$-periodic homotopy theory}

We begin with the stable picture.  $v_n$-periodic stable homotopy theory has its own notion of homotopy and homology groups.  The appropriate homology theory is the \emph{$n$th Morava $K$-theory spectrum} $K(n)$, with
$$ K(n)_* = \FF_p[v_n^{\pm}], \quad \abs{v_n} = 2(p^n-1) $$  
(for $n = 0$ we have $K(0) = H\QQ$ and $v_0 = p$).  The appropriate notion of homotopy groups are the \emph{$v_n$-periodic homotopy groups}, defined as follows.  A finite $p$-local spectrum $V$ is called \emph{type} $n$ if it is $K(n-1)$-acyclic, and not $K(n)$-acyclic.  The periodicity theorem of Hopkins-Smith \cite{HopkinsSmith} states that $V$ has an asymptotically unique \emph{$v_n$ self-map}: a $K(n)$-equivalence  
$$ v: \Sigma^k V \rightarrow V $$
(with $k > 0$ if $n > 0$).  The $v_n$-periodic homotopy groups (with coefficients in $V$) of a spectrum $Z$ are defined to be
$$ v_n^{-1} \pi_* (Z;V) := v^{-1} [\Sigma^* V, Z]_{\Sp}. $$
For $n > 0$ these groups are periodic, of period dividing $k$, the degree of the chosen self-map $v$.
Note these groups do not depend on the choice of $v_n$ self-map (by asymptotic uniqueness) but they do depend on the choice of finite type $n$ spectrum $V$.  However, for any two such spectra $V$, $V'$, it turns out that a map is a $v_n^{-1}\pi_*(-;V)$ isomorphism if and only if it is a $v_n^{-1}\pi_*(-;V')$ isomorphism.  It is straightforward to check that if we take $T(n)$ to be the ``telescope''
$$ T(n) = v_n^{-1} V := \hocolim( V \xrightarrow{v} \Sigma^{-k} V \xrightarrow{v} \Sigma^{-2k} V \xrightarrow{v} \cdots ) $$
then a $v_n^{-1}\pi_*$-isomorphism is the same thing as a $T(n)_*$-isomorphism.

For maps of spectra it can be shown that 
$$ v_n^{-1}\pi_*\text{-isomorphism} \Rightarrow K(n)_*\text{-isomorphism}. $$
Ravenel's \emph{telescope conjecture} \cite{Ravenel84} predicts the converse is true.  This is easily verified in the case of $n = 0$, and 
deep computational work of Mahowald \cite{Mahowaldbo} and Miller \cite{MillerASS} implies the conjecture is valid for $n = 1$. 
It is believed to be false for $n \ge 2$, but the problem remains open despite the valiant efforts of many researchers \cite{MRS}. 

As such, there are potentially two \emph{different} stable $v_n$-periodic categories, $\Sp_{T(n)}$ and $\Sp_{K(n)}$, corresponding to the localizations with respect to the two potentially different notions of equivalence.  $K(n)$-localization gives a functor
$$ (-)_{K(n)}: \Ho(\Sp_{T(n)}) \rightarrow \Ho(\Sp_{K(n)}). $$

\begin{rmk}
Arguably localization with respect to $T(n)$ is more fundamental, but there are no known computations of $\pi_* Z_{T(n)}$ for a finite spectrum $Z$ and $n \ge 2$ (if we had such a computation, we probably would have resolved the telescope conjecture for that prime $p$ and chromatic level $n$).  By contrast, the whole motivation of the chromatic program is that the homotopy groups $\pi_* Z_{K(n)}$ are essentially computable (though in practice these computations get quite involved, and little has been done for $n \ge 3$).
\end{rmk}

\subsection*{The stable chromatic tower}

$p$-local stable homotopy types are assembled from the stable $v_n$-periodic categories in the following manner.
Let $L_n^f\Sp$ denote the category of spectra which are $\bigoplus_{i = 0}^n v_i^{-1} \pi_*$-local, and let $L_n\Sp$ denote the category of spectra which are $\bigoplus_{i = 0}^n K(i)_*$-local, with associated (and potentially different) localization functors $L^f_n$, $L_n$.  A spectrum $Z$ has two potentially different \emph{chromatic towers}
\begin{gather*}
\cdots \rightarrow L^f_2Z \rightarrow L^f_1 Z \rightarrow L^f_0Z, \\
\cdots \rightarrow L_2Z \rightarrow L_1 Z \rightarrow L_0Z.
\end{gather*}
Under favorable circumstances (for example, when $Z$ is finite \cite{HopkinsRavenel}) we have chromatic convergence: the map
$$ Z_{(p)} \rightarrow \holim_n L_n Z $$
is an equivalence.  Presumably one can expect similar results for $L_n^f$, though the authors are not aware of any work on this.

The \emph{monochromatic layers} are the fibers
\begin{gather*}
M_n^fZ \rightarrow L^f_n Z \rightarrow L^f_{n-1} Z, \\
M_nZ \rightarrow L_n Z \rightarrow L_{n-1} Z. \\
\end{gather*}
Let $M_n^f\Sp$ (respectively $M_n\Sp$) denote the subcategory of $L^f_n \Sp$ (respectively $L_n\Sp$) consisting of the image of the functor $M_n^f$
(respectively $M_n$).
Then the pairs of functors
\begin{gather*}
(-)_{T(n)}: \Ho(M_n^f\Sp) \leftrightarrows \Ho(\Sp_{T(n)}) : M_n^f, \\
(-)_{K(n)}: \Ho(M_n\Sp) \leftrightarrows \Ho(\Sp_{K(n)}) : M_n
\end{gather*}
give equivalences between the respective homotopy categories (see, for example, \cite{Bousfield}).  We have
$$ v_n^{-1} V \simeq M^f_n V \simeq V_{T(n)} $$
and
$$
v_n^{-1} \pi_*(Z;V) \cong [\Sigma^* M^f_n V, M_n^f Z]_{\Sp} \cong [\Sigma^* V_{T(n)}, Z_{T(n)}]_{\Sp}. $$

\subsection*{$T(n)$-local Tate spectra}

For $G$ a finite group, and $Z$ a spectrum with a $G$-action, there is a natural transformation
$$ N: Z_{hG} \rightarrow Z^{hG} $$
called the \emph{norm map} \cite{GreenleesMay}.  The cofiber is called the \emph{Tate spectrum}:
$$ Z^{tG} := \mr{cof}(Z_{hG} \rightarrow Z^{hG}). $$
The following theorem is due to Hovey-Sadofsky \cite{HoveySadofsky} in the $K(n)$-local case, and was strengthened by Kuhn \cite{KuhnTate} to the $T(n)$-local case (see also \cite{MahowaldShick}, \cite{GreenleesSadofsky}, and \cite{ClausenMathew}).

\begin{thm}[Greenlees-Sadofsky, Kuhn]\label{thm:tate}
If $Z$ is $T(n)$-local, then the spectrum $Z^{tG}$ is $T(n)$-acyclic, and the norm map is a $T(n)$-equivalence.
\end{thm}

In the case of $n = 0$, this reduces to the familiar statement that rationally, invariants and coinvariants with respect to a finite group are isomorphic via the norm.  In general, this theorem implies that  $T(n)$-local coalgebras, $T(n)$-locally, admit unique divided power structures. In some sense, Theorem~\ref{thm:tate} will be the primary mechanism which will allow unstable $v_n$-periodic homotopy types to admit Lie algebra models.

\subsection*{Unstable $v_n$-periodic homotopy theory}

Perhaps the most illuminating approach to \emph{unstable} $v_n$-periodic homotopy theory is that of \cite{Bousfield}, which we follow here.  This approach builds on previous work of Davis, Mahowald, Dror Farjoun, and many others.  Like the stable case, there will be two potentially different notions of unstable $v_n$-periodic equivalence: one based on unstable $v_n$-periodic homotopy groups, and one based on $K(n)$-homology.

The appropriate unstable analogs of $v_n$-periodic homotopy groups are defined as follows.  The periodicity theorem implies that unstably, a finite type $n$ complex admits a $v_n$-self map
$$ v: \Sigma^{k(N_0+1)} V \rightarrow \Sigma^{kN_0} V $$
for some $N_0 \gg 0$.
For any $X \in \Top_*$, its $v_n$-periodic homotopy groups (with coefficients in $V$) are defined by
$$ v_n^{-1}\pi_*(X;V) := v^{-1}[\Sigma^* V, X]_{\Top_*}. $$
for $n > 0$ ($v_0$-periodic homotopy is taken to be rational homotopy).  For $n > 0$ this definition only makes sense for $* \gg 0$, but because the result is $k$-periodic, one can define these groups for all $* \in \ZZ$.  These give the notion of a \emph{$v_n^{-1}\pi_*$-equivalence of spaces}.  Bousfield argues in \cite{Bousfield} that the appropriate notion of unstable $v_n$-periodic homology equivalence is that of a \emph{virtual $K(n)$-equivalence} --- a map of spaces $X \rightarrow Y$ for which the induced map
$$ (\Omega X)^{\ge n+3} \rightarrow (\Omega Y)^{\ge n+3} $$
is a $K(n)_*$-isomorphism.\footnote{A variant of this definition is explored by Kuhn in \cite{KuhnAdv}.}  Rather than try to explain why this is the appropriate notion we will simply point out that Bousfield proves that if the telescope conjecture is true, then virtual $K(n)$-equivalences are $v_n^{-1}\pi_*$-isomorphisms.

We will focus on the version of unstable $v_n$-periodic homotopy theory based on  $v_n^{-1}\pi_*$-equivalences.  The authors do not know if any attempt has been made to systematically study the unstable theory based on virtual $K(n)$-equivalences (in case the telescope conjecture is false).

Bousfield defines $L_n^f\Top_*$ to be the nullification of $\Top_*$ with respect to 
$$ \Sigma V_{n+1} \vee \bigvee_{\ell \ne p} M(\ZZ/\ell, 2), $$
where $V_{n+1}$ is a type $n+1$ complex of minimal connectivity (say it is $(d_n-3)$-connected).  Let $L_n^f$ denote the associated localization functor.
When restricted to $\Top_*^{\ge d_n}$,  
$L_n^f$ is localization with respect to $\bigoplus_{i = 0}^n v_i^{-1}\pi_*$-equivalences.  For a space $X$ there is an unstable chromatic tower
$$
\cdots \rightarrow L^f_2X \rightarrow L^f_1 X \rightarrow L^f_0X.
$$
The unstable chromatic tower actually always converges to $X_{(p)}$ for a trivial reason: the sequence $d_n$ is non-decreasing and unbounded \cite{Bousfield94}.

The $n$th monochromatic layer is defined to be the homotopy fiber
$$ M_n^fX \rightarrow L^f_n X \rightarrow L^f_{n-1} X. $$
Bousfield defines the $n$th unstable monochromatic category $M_n^f \Top_*$ to be the full subcategory of $\Top_*$ consisting of the spaces of the form $(M_n^fX)^{\ge d_n}$.  Bousfield's work in \cite{Bousfield} implies the equivalences in $M_n^f\Top_*$ are precisely the $v_n^{-1}\pi_*$-equivalences.  Furthermore, for any type $n$ complex $V$ with an unstable $v_n$-self map
$$ v: \Sigma^k V \rightarrow V $$
the $v_n$-periodic homotopy groups are in fact the $V$-based homotopy groups as computed in $\Ho(M_n^f\Top_*)$:
$$ v_n^{-1}\pi_*(X;V) \cong [\Sigma^* V, M_n^f X]_{\Top_*}. $$

\subsection*{The Bousfield-Kuhn functor}

Bousfield and Kuhn \cite{Kuhn}, \cite{Bousfield} observe $v_n$-periodic homotopy groups are the homotopy groups of a spectrum $\Phi_V(X)$. The $kN$th space of this spectrum is given by
$$ \Phi_V(X)_{kN} = \ul{\Top}_*(V,X) $$ 
with spectrum structure maps generated by the maps
$$ \Phi_V(X)_{kN} = \ul{\Top}_*(V,X) \xrightarrow{v^*} \ul{\Top}_*(\Sigma^{k}V,X) \simeq \Omega^k \Phi_V(X)_{k(N+1)}. $$ 
It follows that 
$$ \pi_*\Phi_V(X) \cong v_n^{-1}\pi_*(X;V). $$
The above definition only depended on $\Sigma^{kN}V$ for $N$ large.  As a result, it only depends on the stable homotopy type $\Sigma^\infty V$.  One can therefore take a suitable inverse system $V_i$ of finite type $n$ spectra so that
$$ \holim_i v_n^{-1} V_i \simeq S_{T(n)}. $$
The \emph{Bousfield-Kuhn} functor 
$$ \Phi_n: \Ho(\Top_*) \rightarrow \Ho(\Sp_{T(n)}) $$ 
is given by
$$ \Phi_n (X) = \holim_i \Phi_{V_i^\vee}(X). $$  
We \emph{define} the completed unstable $v_n$-periodic homotopy groups (without coefficients in a type $n$ complex) by\footnote{These should not be confused with the ``uncompleted'' unstable $v_n$-periodic homotopy groups studied by Bousfield, Davis, Mahowald, and others.  These are given as the homotopy groups of $M_n^f\Phi_n(X)$ (see \cite{Kuhncalc}).}
$$ v_n^{-1} \pi_*(X)^{\wedge} := \pi_* \Phi_n(X). $$
The Bousfield-Kuhn functor enjoys many remarkable properties:
\begin{enumerate}
\item For $X \in \Top_*$ and a type $n$ spectrum $V$ we have
$$ [\Sigma^* V, \Phi_n(X)]_{\Sp} \simeq v_n^{-1}\pi_*(X;V). $$ 

\item $\Phi_n$ preserves fiber sequences. 

\item For $Z \in \Sp$ there is a natural equivalence 
$$ \Phi_n \Omega^\infty Z \simeq Z_{T(n)}. $$
\end{enumerate}
Property~(3) above is the strangest property of all: it implies (since by (2) $\Phi_n$ commutes with $\Omega$) that a $T(n)$-local spectrum is \emph{determined} by any one of the spaces in its $\Omega$-spectrum, \emph{independent} of the infinite loop space structure.

\subsection*{Relation between stable and unstable $v_n$-periodic homotopy}

The category $\Ho(\Sp_{T(n)})$ serves as the ``stable homotopy category'' of the unstable $v_n$-periodic homotopy category $\Ho(M_n^f\Top_*)$, with adjoint functors \cite{Bousfield}
$$ (\Sigma^\infty-)_{T(n)}: \Ho(M^f_n\Top_*) \leftrightarrows \Ho(\Sp_{T(n)}): (\Omega^\infty M_n^f-)^{\ge d_n}. $$
Analogously to the rational situation, it is shown in \cite{Bousfield} that the composite
$$ \Ho(\Sp_{T(n)}) \xrightarrow{(\Omega^\infty M_n^f-)^{\ge d_n}} \Ho(M_n^f\Top_*) \xrightarrow{\Phi_n} \Ho(\Sp_{T(n)}) $$
is naturally isomorphic to the identity functor.  Thus the stable $v_n$-periodic homotopy category admits a fully faithful embedding into the unstable $v_n$-periodic homotopy category.  This leads one to expect that there is a ``Lie algebra'' model of unstable $v_n$-periodic homotopy, where the infinite loop spaces correspond to the Lie algebras with trivial Lie structure. 

\subsection*{The $K(n)$-local variant}

There is a variant of the Bousfield-Kuhn functor 
$$ \Phi_{K(n)}: \Ho(\Top_*) \rightarrow \Ho(\Sp_{K(n)}) $$
defined by
$$ \Phi_{K(n)}(X) \simeq \Phi_n(X)_{K(n)}. $$
We then have
$$ \Phi_{K(n)} \Omega^\infty Z \simeq Z_{K(n)}. $$
There is a corresponding variant of completed unstable $v_n$-periodic homotopy groups which (probably to the chagrin of many) we will denote:
$$ v_{K(n)}^{-1}\pi_*(X)^{\wedge} := \pi_*\Phi_{K(n)}(X). $$
Of course if the telescope conjecture is true, $\Phi_{n}(X) \simeq \Phi_{K(n)}(X)$, and the two versions of unstable $v_n$-periodic homotopy agree.  If the telescope conjecture is not true, the groups $v_{K(n)}^{-1}\pi_*$ will likely be far more computable than $v^{-1}_n \pi_*$. 

\section{The comparison map}\label{sec:comparison}

Motivated by rational and $p$-adic homotopy theory, one could ask: to what degree is an unstable homotopy type $X \in M_n^f\Top_*$ modeled by the $T(n)$-local $\mr{Comm}$-algebra $S_{T(n)}^{X}$ (the ``$S_{T(n)}$-valued cochains'')?  I.e., what can be said of the functor:
$$ S_{T(n)}^{(-)}: \Ho(M_n^f\Top_*)^{op} \rightarrow \Ho(\Alg_{\mr{Comm}}({\Sp_{T(n)}}))? $$



The first thing to check is to what degree the unstable $v_n$-periodic homotopy groups of $X$ can be recovered from the algebra $S^X_{T(n)}$:
i.e. for an unstable type $n$ complex $V$ with $v_n$-self map 
$$ v: \Sigma^k V \rightarrow V $$
what can be said of the following composite?
\begin{equation}\label{eq:precomparison}
v_n^{-1}\pi_*(X;V)  \cong [\Sigma^* V, M^f_n(X)]_{\Top_*} \\ 
 \rightarrow \left[S_{T(n)}^X,S^{\Sigma^* V}_{T(n)}\right]_{\Alg_{\mr{Comm}}}
\end{equation}

We begin with the observation, which we learned from Mike Hopkins, that the $\mr{Comm}$-algebra $S^V_{T(n)}$ is actually an ``infinite loop object'' in the category $\Alg_{\mr{Comm}}$:

\begin{prop}
There is an equivalence of $\mr{Comm}$-algebras
$$ S^V_{T(n)} \simeq \mr{triv} (V^\vee). $$
\end{prop} 

\begin{proof}
The existence of the $v_n$-self map $v$ shows that $S_{T(n)}^{V}$ is an infinite loop object of $\Alg_\mr{Comm}$:
$$ S_{T(n)}^{V} \xrightarrow[\simeq]{(v^N)^*} S_{T(n)}^{\Sigma^{Nk} V} \simeq \Omega^{Nk} S_{T(n)}^V. $$
The result follows from the fact that the infinite loop objects in $\Alg_\mr{Comm}$ are the trivial algebras on the underlying spectra.
\end{proof}

Using Corollary~\ref{cor:TAQ}, we now deduce:

\begin{cor}
We have
$$ \ul{\Alg}_\mr{Comm}(S_{T(n)}^X, S_{T(n)}^{\Sigma^* V}) \simeq \Omega^\infty \Sigma^* \TAQ_{S_{T(n)}}(S^X_{T(n)}) \wedge V^\vee. $$
\end{cor}

We deduce that (\ref{eq:precomparison}) refines to a natural transformation
$$ c^V_X: \Phi_{V}(X) \rightarrow \TAQ_{S_{T(n)}}(S_{T(n)}^X) \wedge V^\vee. $$
Taking a suitable homotopy inverse limit of these natural transformations gives a natural transformation
$$ c_X: \Phi_n(X) \rightarrow \TAQ_{S_{T(n)}}(S_{T(n)}^{X}) $$
which we will call \emph{the comparison map}.  A variant, which involves replacing $S_{T(n)}$ with $S_{K(n)}$, everywhere, is defined in \cite{BKTAQ}:
$$ c_X^{K(n)}: \Phi_{K(n)}(X) \rightarrow \TAQ_{{S_{K(n)}}}(S_{K(n)}^X). $$

The main theorem of \cite{BKTAQ} is

\begin{thm}\label{thm:main}
The comparison map $c^{K(n)}_X$ is an equivalence for $X$ a sphere.
\end{thm}

It follows formally from this theorem that the comparison map is an equivalence for a larger class of spaces: the class of finite $\Phi_{K(n)}$-\emph{good} spaces.  This will be discussed in Section~\ref{sec:consequences}.  In the case of $n = 1$, Theorem~\ref{thm:main} was originally proven by French \cite{French}.

It is shown in \cite{Ching} that cobar constructions for $\mc{O}$-coalgebras get a $C\mc{O}$-algebra structure (where $C$ denotes the cooperadic cobar construction). 
The spectrum 
$$ \TAQ_{S_{T(n)}}(S^X_{T(n)}) $$ 
is therefore an algebra over $s^{-1} \Lie_S$ (see Example~\ref{ex:spectrallie}).  We might regard this as a candidate for a ``Lie algebra model'' for the unstable $v_n$-periodic homotopy type of $X$, though this is probably only reasonable for $X$ finite, as will be explained in Section~\ref{sec:Heuts}.

\section{Outline of the proof of the main theorem}\label{sec:proof}

Our approach to Theorem~\ref{thm:main} is essentially computational in nature, and uses the Morava $E$-theory Dyer-Lashof algebra in an essential way.  Unfortunately, the proof given in \cite{BKTAQ} is necessarily technical, and consequently is not optimized for leisurely reading.  In this section we give an overview of the main ideas of our proof.  As we will explain in Sections~\ref{sec:AroneChing} and \ref{sec:Heuts}, Arone-Ching \cite{AroneChingvn} and Heuts \cite{Heuts2} have announced more abstract approachs to prove Theorem~\ref{thm:main}, with stronger consequences.  Perhaps the situation is comparable to the early work on $p$-adic homotopy theory of Kriz and Goerss \cite{Kriz}, \cite{Goerss}: Kriz's approach (like that of \cite{Mandell}) is computational, based on the Steenrod algebra, whereas Goerss' is abstract, based on Galois descent and model category theory.  Both approaches offer insight into the theory of using commutative algebras/coalgebras to model $p$-adic homotopy types.  We hope the same is true of the two approaches to model unstable $v_n$-periodic homotopy.

\subsection*{Goodwillie towers}

The proof of \ref{thm:main} involves induction up the Goodwillie towers of both the source and target of the comparison map.  The key fact that the argument hinges on is an observation of Kuhn \cite{KuhnMcCord}: the layers of both of these towers are abstractly equivalent.

For our application of Goodwillie calculus to the situation, we point out that, in the context of model categories, Pereira \cite{Pereira1} has shown that Goodwillie's calculus of functors (as developed in \cite{Goodwillie}) applies to homotopy functors
$$ F: \mc{C} \rightarrow \mc{D} $$
between arbitrary model categories with fairly minimal hypotheses (see also \cite{BiedermannRondigs} and \cite{Lurie})\footnote{Yet another general treatment of homotopy calculus can be found in \cite{BauerJohnsonMcCarthy}, but at present this approach only applies to functors which take values in spectra.}.  For simplicity we shall assume that $\mc{C}$ and $\mc{D}$ are pointed, and restrict attention to reduced $F$ (i.e. $F(*) \simeq *$).

Associated to $F$ is its \emph{Goodwillie tower}, a series of \emph{$k$-excisive approximations}
$$ P_kF : \mc{C} \rightarrow \mc{D} $$
which form a tower under $F$:
$$ F \rightarrow \cdots \rightarrow P_kF \rightarrow P_{k-1}F \rightarrow \cdots \rightarrow P_1F. $$
We say the Goodwillie tower \emph{converges} at $X$ if the map
$$ F(X) \rightarrow \holim_k P_kF(X) $$
is an equivalence.
The \emph{layers} of the Goodwillie tower are the fibers
$$ D_kF \rightarrow P_kF \rightarrow P_{k-1}F. $$
If $F$ is finitary (i.e. preserves filtered homotopy colimits), the layers take the form
$$ D_kF(X) \simeq \Omega^\infty_\mc{D} \mr{cr}^{lin}_k(F)(\Sigma^\infty_{\mc{C}}X, \cdots, \Sigma^\infty_\mc{C} X)_{h\Sigma_k} $$
where 
$$ \mr{cr}^{lin}_k(F): \Sp(\mc{C})^{\times k} \rightarrow \Sp(\mc{D}) $$ 
is a certain symmetric multilinear functor called the \emph{multilinearized cross-effect}.
In the case where $\Sp(\mc{C}), \Sp(\mc{D})$ are Quillen equivalent to $\Sp = \Sp(\Top_*)$, the multilinearized cross effect is given by 
$$ \mr{cr}^{lin}_k(F)(Z_1, \cdots, Z_k) \simeq  \partial_kF \wedge Z_1 \wedge \cdots \wedge Z_k $$  
where $\partial_kF$ is a spectrum with $\Sigma_k$-action (the \emph{$k$th derivative of $F$}), and we have
$$ D_kF(X) \simeq \Omega^\infty_{\mc{D}} \left(\partial_kF \wedge_{h\Sigma_k} (\Sigma^\infty_{\mc{C}}X)^{\wedge k}\right). $$
The Goodwillie tower is an analog for functors of the Taylor series of a function, with $D_k(F)$ playing the role of the $k$th term of the Taylor series.

We consider the Goodwillie towers of the functors
\begin{gather*} 
\Phi_{K(n)}: \Top_* \rightarrow \Sp_{K(n)} \\
\TAQ_{{S_{K(n)}}}(S^{(-)}_{K(n)}) 
: \Top_* \rightarrow \Sp_{K(n)}.
\end{gather*}
Note that the second of these functors is not finitary ($\Phi_{K(n)}$ is actually finitary, as long as the corresponding homotopy colimit is taken in the category $\Sp_{K(n)}$).
In the case of $\Phi_{K(n)}$, it is fairly easy to see that its Goodwillie tower is closely related to the Goodwillie tower of the identity functor
$$ \Id:\Top_* \rightarrow \Top_*. $$

\begin{lem}\label{lem:Phitower}
There are equivalences
$$ P_k\Phi_{K(n)} \simeq \Phi_{K(n)}P_k\Id. $$
\end{lem}

\begin{proof}
This follows easily from observing that the fibers of the RHS are given by
$$ \Phi_{K(n)} D_k\Id(X) \simeq (s^{-1}\Lie_k \wedge_{h\Sigma_k} X^{\wedge k})_{K(n)} $$
and are therefore homogeneous of degree $k$. 
\end{proof}

More subtly, Kuhn constructed a filtration on $\TAQ^{R}$ \cite{KuhnMcCord} which results in a tower
\begin{equation}\label{eq:Kuhntower}
 \TAQ_{R}(A) \rightarrow \cdots \rightarrow F_k\TAQ_{R}(A) \rightarrow F_{k-1}\TAQ_{R}(A) \rightarrow \cdots.
 \end{equation}
For all $A$ we have an equivalence
\begin{equation}\label{eq:kuhnconv}
\TAQ_R(A) \xrightarrow{\simeq} \holim F_k \TAQ_R(A)
\end{equation}
for the simple reason that Kuhn's filtration of $\TAQ^R$ is exhaustive. 

\begin{thm}[Kuhn \cite{KuhnMcCord}]\label{thm:McCord}
The fibers of the tower (\ref{eq:Kuhntower}) are given by
$$ s^{-1}\Lie_k \wedge^{h\Sigma_k} (A^{\wedge_R k})^\vee \rightarrow F_k\TAQ_R(A) \rightarrow F_{k-1}\TAQ_R(A). $$
\end{thm}

\begin{cor}\label{cor:TAQtower}
For finite $X$ the Goodwillie tower of the functor $\TAQ_{S_{K(n)}}(S^{(-)}_{K(n)})$ is given by
$$ P_k(\TAQ_{S_{K(n)}}(S^{(-)}_{K(n)}))(X) \simeq F_k \TAQ_{S_{K(n)}}(S^{X}_{K(n)}). $$
\end{cor}

\begin{proof}
Combining Theorem~\ref{thm:tate} with Theorem~\ref{thm:McCord} shows the layers of the RHS are equivalent to
$$ (s^{-1}\Lie_k \wedge_{h\Sigma_k} X^{\wedge k})_{K(n)}. $$
In particular, they are homogeneous of degree $k$.
\end{proof}

It follows that the comparison map actually induces a natural transformation of towers
$$ P_n(c^{K(n)}_{X}): \Phi_{K(n)}P_k\Id(X) \rightarrow F_k\TAQ_{S_{K(n)}}(S_{K(n)}^X) $$
when restricted to finite $X$.  In fact, the proofs of Lemma~\ref{lem:Phitower} and Corollary~\ref{cor:TAQtower} actually imply that for $X$ finite, the layers of these towers are abstractly equivalent.  Thus, to show that the maps $P_n(c^{K(n)}_X)$ are equivalences, we just need to show that they \emph{induce} equivalences on the layers (which we already know are equivalent)!  This will be accomplished computationally using

\subsection*{The Morava $E$-theory Dyer-Lashof algebra}

Let $E_n$ denote the $n$th Morava $E$-theory spectrum, with 
$$ (E_n)_* \cong W(\FF_{p^n})[[u_1, \ldots, u_{n-1}]][u^\pm]. $$
The ring $(E_n)_0$ has a unique maximal ideal $\mf{m}$.  We shall let 
$$ (E^\wedge_n)_*Z := \pi_*(E_n \wedge Z)_{K(n)} $$
denote the completed $E$-homology of a spectrum $Z$.  If the uncompleted Morava $E_n$-homology is flat over $(E_n)_*$, the completed $E$-homology is the $\mf{m}$-completion of the uncompleted homology.  Let $K_n$ denote the $2$-periodic version of $K(n)$, with
$$ (K_n)_* \cong (E_n)_*/\mf{m} \cong \FF_{p^n}[u^\pm]. $$ 

In \cite{Rezk}, the second author defined a monad\footnote{The monad denoted $\mb{T}$ here is actually a non-unital variant of the monad $\mb{T}$ of \cite{Rezk}.} 
$$ {\mb{T}}: \Mod_{(E_n)_*} \rightarrow \Mod_{(E_n)_*} $$
such that the completed $E$-homology of a $\mr{Comm}$-algebra has the structure of a $\mb{T}$-algebra.  A $\mb{T}$-algebra is basically an algebra over the Morava $E$-theory Dyer-Lashof algebra $\Gamma_n$.  For an $(E_n)_*$-module $M$, the value of the functor $\mb{T} M$ is the free $\Gamma_n$-algebra on $M$ (for a precise description of what is meant by this, consult \cite{Rezk}).

The work of Strickland \cite{Strickland} basically determines the structure of the dual of $\Gamma_n$ in terms of rings of functions on the formal schemes of subgroups of the Lubin-Tate formal groups.
In the case of $n = 1$, the corresponding Morava $E$-theory is $p$-adic $K$-theory, and $\Gamma_1$ is generated by the Adams operation $\psi^p$ with no relations.  In the case of $n = 2$, the explicit structure of $\Gamma_2$ was determined by the second author in \cite{Rezk08} for $p = 2$, and mod $p$ for all primes in \cite{RezkMIC}.  An integral presentation of $\Gamma_2$ has recently been determined by Zhu \cite{Zhu1}.  Very little is known about the explicit structure of $\Gamma_n$ for $n \ge 3$ except that it is Koszul \cite{RezkKoszul} in the sense of Priddy \cite{Priddy}.

For the purpose of our discussion of Theorem~\ref{thm:main}, the only thing we really need to know about $\mb{T}$ is the following theorem of the second author (see \cite{Rezk}):

\begin{thm}
If $(E_n^\wedge)_*Z$ is flat over $(E_n)_*$, then 
the natural transformation
$$ {\mb{T}} (E_n^\wedge)_*Z \rightarrow (E^\wedge_n)_*\mc{F}_{\mr{Comm}} Z $$
induces an isomorphism 
$$ ({\mb{T}} (E_n^\wedge)_*Z)^\wedge_\mf{m} \xrightarrow{\cong} (E_n^\wedge)_*\mc{F}_{\mr{Comm}} Z. $$
\end{thm}

There is a ``completed'' variant of the functor $\mc{F}_{\Comm}$:
$$ \widehat{\mc{F}}_\Comm (Z) := \prod_i Z^{i}_{h\Sigma_i}. $$
The following lemma of \cite{BKTAQ} is highly non-trivial, as completed Morava $E$-theory in general behaves badly with respect to products.

\begin{lem}
There is a completed variant of the free $\mb{T}$-algebra functor: 
$$ \widehat{\mb{T}}: \Mod_{E_*} \rightarrow \Alg_\mb{T} $$
and for spectra $Z$ a natural transformation
$$ \widehat{\mb{T}}(E_n^\wedge)_*Z \rightarrow (E_n^\wedge)_*\widehat{\mc{F}}_\Comm Z $$
which is an isomorphism if $(E_n^\wedge)_*Z$ is flat and finitely generated. \end{lem}

In \cite{BKTAQ} we construct a version of the Basterra spectral sequence for $E$-theory: for a $K(n)$-local $\mr{Comm}$-algebra $A$ whose $E_n$-homology satisfies a flatness hypothesis, the spectral sequence takes the form
\begin{equation}\label{eq:BSS}
AQ_{{\mb{T}}}^{*,*}((E_n^\wedge)_*A;(K_n)_*) \Rightarrow (K_n)_*\TAQ_{S_{K(n)}}(A).
\end{equation}
Here $AQ_{\mb{T}}^{*,*}(-;M)$ denotes Andre-Quillen cohomology of ${\mb{T}}$-algebras with coefficients in an $E_*$-module $M$ (see \cite{BKTAQ} for a precise definition --- these cohomology groups are closely related to those defined in \cite{GoerssHopkins}). 

\subsection*{The comparison map on $QX$}

The next step in the proof of Theorem~\ref{thm:main} is to prove the following key proposition.

\begin{prop}\label{prop:key}
There is a non-negative integer $N$ so that for all $N$-fold suspension spaces $X$ with $(E_n^\wedge)_*X$ free and finitely generated over $(E_n)_*$, the comparison map
$$ (\Sigma^\infty X)_{K(n)} \simeq \Phi_{K(n)}(QX) \xrightarrow{c^{K(n)}_{QX}} \TAQ_{S_{K(n)}}(S_{K(n)}^{QX}) $$
is an equivalence. 
\end{prop}

We will prove this proposition by showing that the comparison map induces an isomorphism in Morava $K(n)$-homology. The first step is to compute the $K(n)_*$-homology of the the RHS.  This is accomplished in \cite{BKTAQ} with the following technical lemma:

\begin{lem}\label{lem:technical}
For $X$ satisfying the hypotheses of Proposition~\ref{prop:key}, there is a map of $(E_n)_*$-modules
$$ (E^\wedge_n)_* S_{K(n)}^{QX} \rightarrow \widehat{\mb{T}} \tilde{E}_n^*X $$
which is an \emph{isomorphism of ${\mb{T}}$-algebras mod $\mf{m}$}, in the sense that it is an isomorphism mod $\mf{m}$, and commutes with the ${\mb{T}}$-action mod $\mf{m}$. 
\end{lem}

Heuristically, this lemma might seem to follow from Theorem~\ref{thm:tate} and the Snaith splitting:
\begin{align*}
S_{K(n)}^{QX} 
& \simeq S_{K(n)}^{\bigvee_i X^{i}_{h\Sigma_i}} \\
& \simeq \prod_i \left( S_{K(n)}^{X^{i}} \right)^{h\Sigma_i} \\
& \simeq \left(\prod_i \left( S_{K(n)}^{X^{i}} \right)_{h\Sigma_i} \right)_{K(n)} \\
& \simeq \left( \widehat{\mc{F}}_\Comm S_{K(n)}^X \right)_{K(n)}.
\end{align*}
However, as was pointed out to us by Nick Kuhn, this is \emph{not} an equivalence of $\Comm$-algebras (or even non-unital $H_\infty$-ring spectra)!  Nevertheless, Lemma~\ref{lem:technical} establishes that on Morava $E$-theory, this sequence of equivalences induces an isomorphism of $\mb{T}$-algebras mod $\mf{m}$. 

\begin{proof}[Proof of Proposition~\ref{prop:key}]
The natural transformation
$$  \Sigma^\infty QX = \Sigma^\infty \Omega^\infty \Sigma^\infty X  \rightarrow \Sigma^\infty X  $$
induces a natural transformation
$$ S_{K(n)}^{X} \rightarrow S_{K(n)}^{QX} $$
of spectra, hence a natural transformation
$$ \mc{F}_{\Comm_{S_{K(n)}}} S_{K(n)}^{X} \rightarrow S_{K(n)}^{QX} $$
of $\Comm$-algebras.  We thus get a natural transformation
\begin{align*}
\TAQ_{S_{K(n)}}(S_{K(n)}^{QX}) \xrightarrow{\eta_X} & \TAQ_{S_{K(n)}}(\mc{F}_{\Comm_{S_{K(n)}}} S_{K(n)}^{X}) \\
\simeq & (\Sigma^\infty X)_{K(n)} \\
\simeq & \Phi_{K(n)}(QX).
\end{align*}
It can be shown that $\eta_X \circ c^{K(n)}_{QX} \simeq \Id$.  Since $(\td{K}_n)_*X$ is finite, it suffices to show that $(K_n)_*\TAQ_{S_{K(n)}}(S_{K(n)}^X)$ is abstractly isomorphic to $(\td{K}_n)_*X$.  This is proven using the Basterra spectral sequence (\ref{eq:BSS}).  The spectral sequence collapses to the desired result as we have (using Lemma~\ref{lem:technical})
\begin{align*}
AQ^{s,*}_\mb{T}((E_n^\wedge)_* S^{QX}_{K(n)}; (K_n)_*)
& \cong AQ^{s,*}_\mb{T}(\widehat{\mb{T}}\td{E}_n^*X; (K_n)_*) \\
& \cong AQ^{s,*}_\mb{T}(\mb{T}\td{E}_n^*X; (K_n)_*) \\
& \cong 
\begin{cases}
(\td{K}_n)_*X, & s = 0, \\
0, & s > 0.
\end{cases}
\end{align*}
\end{proof}

\subsection*{The comparison map on spheres}

We now outline the proof of Theorem~\ref{thm:main}.  Let $X = S^q$.  The following strong convergence theorem of Arone-Mahowald \cite{AroneMahowald} is crucial.

\begin{thm}[Arone-Mahowald]\label{thm:AroneMahowald}
The natural transformation
$$ \Phi_{K(n)}(X) \rightarrow  \Phi_{K(n)} P_{k}\Id(X) $$
is an equivalence for $q$ odd and $k = p^n$, or $q$ even and $k = 2p^n$. 
\end{thm}

The basic strategy is to attempt to apply Proposition~\ref{prop:key} to the Bousfield-Kan cosimplicial resolution
$$ X \rightarrow Q^{\bullet+1}X = \left( QX \Rightarrow QQX \Rrightarrow \cdots \right). $$
We first assume that the dimension $q$ of the sphere $X = S^q$ is large and odd.  Unfortunately, for $s \ge 1$,  $Q^s X$ does not satisfy the finiteness hypotheses of Proposition~\ref{prop:key} required to deduce that the comparison map is an equivalence.
We instead consider the diagram
\begin{equation}\label{diag:main}
\xymatrix{
\Phi_{K(n)}(X) \ar[d]_{c^{K(n)}_X} \ar[r] &
\Tot \Phi_{K(n)} P_{p^n}(Q^{\bullet+1})(X) \ar[d]^{\simeq} \\
\TAQ_{S_{K(n)}}(S^X_{K(n)}) \ar[r] &  
\Tot \TAQ_{S_{K(n)}}\left(S^{P_{p^n}(Q^{\bullet+1})(X)}_{K(n)}\right) 
}
\end{equation}
In the above diagram, the right vertical map is an equivalence using Proposition~\ref{prop:key}: the Snaith splitting may be iterated to give an equivalence \cite{AroneKankaanrintaSnaith}
$$ P_{p^n}(Q^{s+1})(X) \simeq QY^s $$
where the space $Y^s$ does satisfy the hypotheses of Proposition~\ref{prop:key}.
Using finiteness properties of the cosimplicial space $Y^\bullet$, we show in \cite{BKTAQ} that the top horizontal map
$$ \Phi_{K(n)}(X) 
\simeq \Phi_{K(n)} P_{p^n}\Id(X) \rightarrow \Tot \Phi_{K(n)} P_{p^n}(Q^{\bullet+1})(X) $$
of (\ref{diag:main}) is an equivalence.  It follows that the comparison map has a weak retraction when restricted to large dimensional odd spheres $X$:
$$
\xymatrix{
\Phi_{K(n)}X \ar[rr]^\simeq \ar[dr]_{c^{K(n)}_X} && \Phi_{K(n)} X \\
& \TAQ_{S_{K(n)}}(S_{K(n)}^X) \ar[ur] 
} 
$$
Using standard methods of Goodwillie calculus (or more specifically, Weiss calculus in this case) it follows that for $X$ a large dimensional odd sphere, the induced map on Goodwillie towers 
\begin{equation}\label{eq:towercomp}
 \{P_k\Phi_{K(n)}(X)\}_k \xrightarrow{c^{K(n)}} \{F_k \TAQ_{S_{K(n)}}(S_{K(n)}^X)\}_k
 \end{equation}
has a weak retraction.  The theorem (for $X$ a large dimensional odd sphere) follows from the fact that (1) the layers of the towers are abstractly equivalent, and (2) the layers of the towers have finite $K(n)$-homology.
Since Goodwillie derivatives are determined by the values of the functors on large dimensional spheres, it follows that the induced map of symmetric sequences
\begin{equation}\label{eq:partialcomp}
\partial_*\Phi_{K(n)} \xrightarrow{c^{K(n)}} \partial_*(\TAQ_{S_{K(n)}}(S^{(-)}_{K(n)}))
\end{equation}
is an equivalence.  It follows that the map (\ref{eq:towercomp}) is actually an equivalence of towers for \emph{all} spheres $X$.  The theorem now follows from Theorem~\ref{thm:AroneMahowald} and (\ref{eq:kuhnconv}). 

\section{Consequences}\label{sec:consequences}

We begin this section by explaining how our result for spheres actually implies that the comparison map is an equivalance on the larger class of finite $\Phi_{K(n)}$-\emph{good spaces}.  We also survey some computational applications of our theory, and end the section with some questions.

\subsection*{$\Phi_{K(n)}$-good spaces}

We observe that our method of proving Theorem~\ref{thm:main} actually yields a stronger result.

\begin{thm}\label{thm:main2}
For $X$ any finite complex, the comparison map gives an equivalence of towers
$$
 \{P_k\Phi_{K(n)}(X)\}_k \xrightarrow[\simeq]{c^{K(n)}} \{F_k \TAQ_{S_{K(n)}}(S_{K(n)}^X)\}_k
$$
and therefore an equivalence
$$ c^{K(n)}_{X}: P_\infty\Phi_{K(n)}(X) \xrightarrow{\simeq} \TAQ_{S_{K(n)}}(S_{K(n)}^X). $$
\end{thm}

\begin{proof}
This follows from the equivalence (\ref{eq:partialcomp}).  Note the restriction to finite complexes is necessary as the target functor is not finitary.
\end{proof}

We will say that a space $X$ is \emph{$\Phi_{K(n)}$-good} if the map
\begin{equation}\label{eq:vncomp}
 \Phi_{K(n)}(X) \rightarrow \holim_k P_k(\Phi_{K(n)})(X)
 \end{equation}
is an equivalence.

\begin{cor}\label{cor:main}
A finite space $X$ is $\Phi_{K(n)}$-good if and only if the comparison map
$$ c^{K(n)}_X: \Phi_{K(n)}(X) \rightarrow \TAQ_{S_{K(n)}}(S^X_{K(n)}) $$
is an equivalence.
\end{cor}

Theorem~\ref{thm:AroneMahowald} clearly implies spheres are
$\Phi_{K(n)}$-good.  The functor $\Phi_{K(n)}$ preserves all fiber sequences, but it seems the target of the comparison map is not as robust.

\begin{lem}
The functor $\TAQ_{S_{K(n)}}(S^{(-)}_{K(n)})$ preserves products of finite spaces.
\end{lem}

\begin{proof}
This follows from the fact that $\TAQ$ is excisive, together with the fact that there is an equivalence of augmented commutative $S$-algebras 
$$ S^{X \times Y_+} \simeq S^{X_+} \wedge S^{Y_+}. $$
\end{proof}

\begin{cor}
The product of finite $\Phi_{K(n)}$-good spaces is $\Phi_{K(n)}$-good.
\end{cor}

We shall say that a fiber sequence of finite spaces
$$ F \rightarrow E \rightarrow B $$
is \emph{$K(n)$-cohomologically Eilenberg-Moore} if the map of augmented commutative $S$-algebras
$$ S^{E_+} \wedge_{S^{B_+}} S \rightarrow S^{F_+} $$
is a $K(n)$-equivalence.  The motivation behind this terminology is that with this condition the associated cohomological Eilenberg-Moore spectral sequence converges \cite[Sec.~IV.6]{EKMM}
$$ \Tor^{*,*}_{K(n)^*(B)}(K(n)^*(F), K(n)^*) \Rightarrow K(n)^*(E). $$

The following lemma follows immediately from the excisivity of $\TAQ$.

\begin{lem}
Suppose that 
$$ F \rightarrow E \rightarrow B $$
is a fiber sequence of finite spaces which is $K(n)$-cohomologically Eilenberg-Moore. Then the induced sequence
$$ \TAQ_{S_{K(n)}}(S_{K(n)}^F) \rightarrow \TAQ_{S_{K(n)}}(S_{K(n)}^E) \rightarrow \TAQ_{S_{K(n)}}(S_{K(n)}^B) $$
is a fiber sequence.
\end{lem}

Since $\Phi_{K(n)}$ preserves fiber sequences, we deduce the following.

\begin{cor}\label{cor:EMgood}
Suppose that 
$$ F \rightarrow E \rightarrow B $$
is a fiber sequence of finite spaces which is $K(n)$-cohomologically Eilenberg-Moore.  Then if any two of the spaces in the sequence are $\Phi_{K(n)}$-good, so is the third.
\end{cor}

Using this we can give examples of $\Phi_{K(n)}$-good spaces which are not spheres (or finite products of spheres).

\begin{prop}
The special unitary groups $SU(k)$ and symplectic groups $Sp(k)$ are $\Phi_{K(n)}$-good.
\end{prop}

\begin{proof}
For simplicity we treat the special unitary groups; the symplectic case is essentially identical.
Petrie \cite{Petrie} showed that additively there is an isomorphism
$$ MU_*SU(k) \cong \Lambda_{MU_*}[y_3, y_5, \ldots, y_{2k-1}]. $$
It follows from the collapsing universal coefficient spectral sequence that there is an additive isomorphism
\begin{equation}\label{eq:KnSUk}
K(n)^*SU(k) \cong \Lambda_{K(n)_*}[x_3, x_5, \ldots, x_{2k-1}].
\end{equation}
The Atiyah-Hirzebruch spectral sequence for $K(n)^*SU(k)$ must therefore collapse (any differentials would otherwise make the rank of $K(n)^*SU(k)$ too small).  There are no possible extensions, as the exterior algebra is free as a graded-commutative algebra.  Therefore (\ref{eq:KnSUk}) is an isomorphism of $K(n)_*$-algebras.  This can than be used to show that the fiber sequences
$$ SU(k-1) \rightarrow SU(k) \rightarrow S^{2k-1} $$
are $K(n)$-cohomologically Eilenberg-Moore.  The result follows by induction (using Corollary~\ref{cor:EMgood}).
\end{proof}

Not all spaces are $\Phi_{K(n)}$-good.  Brantner and Heuts have recently shown that wedges of spheres of dimension greater than 1, and mod $p$ Moore spaces, are examples of non-$\Phi_{K(n)}$-good spaces \cite{BrantnerHeuts}.

\subsection*{Some computations}

The target of the comparison map should be regarded as computable, and the source should be regarded as mysterious.  Because of this, our theorem has important computational consequences.  We take a moment to mention some things that have already been done.

In \cite{BKTAQ}, we show that the Morava $E$-theory of the layers of the Goodwillie tower for $\Phi_{K(n)}$ evaluated on $S^1$ are given by the cohomology of the second author's \emph{modular isogeny complex} \cite{RezkMIC}.
Theorem~\ref{thm:main2} was applied by the authors in \cite{BKTAQ} to
compute the Morava $E$-theory of the attaching maps between the
consecutive non-trivial layers of this Goodwillie tower.  Iterating
the double suspension, these computations then restrict to give an
approach to computing the Morava $E$-theory of the Goodwillie tower of
$\Phi_{K(n)}$ evaluated on all odd dimensional spheres.  

We envision this as a step in the program of Arone-Mahowald
\cite{AroneMahowald}, \cite{Kuhncalc} to compute the unstable
$v_{K(n)}$-periodic homotopy groups of spheres (and other
$\Phi_{K(n)}$-good spaces) using stable $v_{K(n)}$-periodic homotopy
groups and Goodwillie calculus.  This would generalize a number of
known calculations in the case of $n=1$. These
computations include those of Mahowald \cite{MahowaldImJ} and 
Thompson \cite{Thompson} for spheres, and would generalize Bousfield's
technology \cite{Bousfield99}, \cite{Bousfield05}, \cite{Bousfield07},
for computations for spherically resolved spaces.  Bousfield's theory was applied successfully by Don Davis and his collaborators to compute
$v_1$-periodic homotopy groups of various compact Lie groups (see \cite{Davis}, where the previous work on this subject, by Bendersky, Davis, Mahowald, and Mimura is summarized\footnote{Technically, the previous computations used the unstable Adams-Novikov spectral sequence, but were simplified using Bousfield's results.}).  

To this end, Zhu has used his explicit computation of the Morava $E$-theory Dyer-Lashof algebra at $n = 2$ \cite{Zhu1} to compute the Morava $E$-theory of $\Phi_{K(2)}(S^q)$ for $q$ odd \cite{Zhu2}.

Using our technology, but employing $BP$-theory instead of Morava $E$-theory, Wang has computed the groups $v^{-1}_{K(2)}\pi_*(S^3)^{\wedge}$ for $p \ge 5$ \cite{Wang}.  Wang has also computed the monochromatic Hopf invariants of the $\beta$-family at these primes.  These are the analogs of the classical Hopf invariants, but computed in the category $M_2^f\Top_*$.

\begin{thm}[Wang \cite{WangHopf}]
The monochromatic Hopf invariant of $\beta_{i/j,k}$ is $\beta_{i-j/k}$.
\end{thm}

Finally, Brantner has recently computed the algebra of power operations which naturally act on the completed $E$-theory of any spectral Lie algebra (such as those arising as spectral Lie algebra models of unstable $v_n$-periodic homotopy types) \cite{Brantner}.

\subsection*{Some questions}

We end this section with some questions.

\begin{ques} 
Does the bracket from the $s^{-1}\Lie$-structure on $\TAQ$-coincide with the Whitehead product in unstable $v_n$-periodic homotopy?
\end{ques}

\begin{ques}
In \cite{Bousfield07}, Bousfield introduces the notion of a $\widehat{K}\Phi$-good space.  What is the relationship between this notion and the notion of being $\Phi_{K(1)}$-good?
\end{ques}

\begin{ques}
Is there a relationship to $X$ being $\Phi_{K(n)}$-good and the convergence of $X$'s unstable $v_n$-periodic $E^\wedge_n$-based Adams spectral sequence to $v_{K(n)}^{-1}\pi_*(X)^{\wedge}$?
\end{ques}

\section{The Arone-Ching approach}\label{sec:AroneChing}

The central component of Goodwillie's theory of homotopy calculus, from which the theory derives much of its computational power, is the idea that the layers of the Goodwillie tower of a functor $F$ are classified by its symmetric sequence of derivatives $\partial_* F$.  Arone and Ching have pursued a research program which seeks to endow $\partial_* F$ with enough extra structure to recover the entire Goodwillie tower of $F$ \cite{AroneChingchain}, \cite{AroneChing}, \cite{AroneChingcross}.  In this section we will focus on the setup of \cite{AroneChingchain}, and will describe their approach to give a conceptual alternative proof of Theorem~\ref{thm:main2}.  \emph{In this section we will only consider homotopy functors}
$$ F: \mc{C} \rightarrow \mc{D} $$
\emph{where $\mc{C}$ and $\mc{D}$ are either the categories of pointed spaces or spectra.}\footnote{Later in this section we will also allow $\mc{D}$ to be $\Sp_{T(n)}$.}

\subsection*{Modules over operads}

Let $\mc{O}$ be a reduced operad in $\Mod_R$, and let $\mc{A} = \{ \mc{A}_i \}$ be a symmetric sequence of $R$-module spectra.  A left (respectively right) module structure on $\mc{A}$ is the structure of an associative action
$$ \mc{O} \circ \mc{A} \rightarrow \mc{A} \qquad (\mr{resp.} \: \mc{A} \circ \mc{O} \rightarrow \mc{A}). $$
One similarly has the notion of a left/right comodule structure.  Explicitly, a left $\mc{O}$-module structure on $\mc{A}$ is encoded in structure maps
$$ \mc{O}_k \wedge_R \mc{A}_{n_1} \wedge_R \cdots \wedge_R \mc{A}_{n_k} \rightarrow \mc{A}_{n_1+\cdots +n_k}, $$
and a right $\mc{O}$-module structure is encoded in structure maps
$$ \mc{A}_k \wedge_R \mc{O}_{n_1} \wedge_R \cdots \wedge_R \mc{O}_{n_k} \rightarrow \mc{A}_{n_1+\cdots+n_k}. $$
The structure maps for left/right comodules are obtained simply by reversing the direction of the above arrows.

Suppose that $A$ is an $\mc{O}$-algebra.  Regarding $A$ as the symmetric sequence
$$ (A, \ast, \ast, \cdots ) $$
with $A$ in the $0$th spot, the $\mc{O}$-algebra structure on $A$ can also be regarded as a left $\mc{O}$-module structure on $A$.
Less obviously, the $\mc{O}$-algebra structure can also be encoded in a right \emph{comodule} structure on the symmetric sequence\footnote{It is more natural to define the $0$th space of the symmetric sequence $A^{\wedge_R *}$ to be $R$, but it makes no difference as we are assuming $\mc{O}$ is reduced.  For the purposes of the rest of the section this convention will be more useful.}
$$ A^{\wedge_R *} := (\ast, A, A^2, A^3, \cdots ). $$
For simplicity, assume that each of the $R$-module spectra $\mc{O}_i$ are strongly dualizible.  Then $\mc{O}^\vee$ is a cooperad, and the $\mc{O}$-algebra structure on $\mc{A}$ is encoded in a right $\mc{O}^\vee$-comodule structure on $A^{\wedge_R *}$  
$$ A^{n_1+\cdots+n_k} \rightarrow A^k \wedge_R \mc{O}^\vee_{n_1} \wedge_R \cdots \wedge_R \mc{O}^\vee_{n_k}. $$
These comodule structure maps are adjoint to the maps
$$ \mc{O}_{n_1} \wedge_R \cdots \wedge_R \mc{O}_{n_k} \wedge_R A^{n_1+\cdots +n_k} \rightarrow A^k $$
obtained by smashing together $k$ algebra structure maps.

\subsection*{Koszul duality, again}

\emph{In this subsection, all symmetric sequences $\mc{A}$ are assumed to satisfy $\mc{A}_0 = \ast.$}  With this hypothesis, Ching's construction of the cooperad structure on the operadic bar construction
$$ B\mc{O} = B(1_R, \mc{O}, 1_R) $$
extends to give $B\mc{O}$-comodule structures \cite{Ching}.  Specifically, suppose that $\mc{M}$ is a right $\mc{O}$-module.  Then 
$$ B\mc{M} := B(\mc{M}, \mc{O}, 1_R) $$
gets the structure of a right $B\mc{O}$-comodule.  Similarly, for a left $\mc{O}$-module $\mc{N}$,
$$ B\mc{N}:= B(1_R, \mc{O}, \mc{N}) $$
gets the structure of a left $B\mc{O}$-comodule.  There are dual statements which endow cobar constructions of comodules with module structures.

In this manner the operadic bar and cobar constructions give functors
\begin{align*}
 B: \mr{lt.Mod}_\mc{O} & \leftrightarrows \mr{lt.Comod}_{B\mc{O}}: C, \\
 B: \mr{rt.Mod}_\mc{O} & \leftrightarrows \mr{rt.Comod}_{B\mc{O}}: C.
\end{align*}

Some of the key ideas in the following Koszul duality theorem can be found in \cite{AroneChingchain}, but a proof of the full statement should appear in \cite{ChingKoszul}.

\begin{thm}[Ching]\label{thm:Koszul}
The bar/cobar constructions give an equivalence of homotopy categories of right (co)modules
$$ B: \Ho(\mr{rt.Mod}_\mc{O}) \leftrightarrows \Ho(\mr{rt.Comod}_{B\mc{O}}): C. $$
In the case of left modules, the bar construction gives a fully faithful embedding
$$ B: \Ho(\mr{lt.Mod}_\mc{O}) \hookrightarrow \Ho(\mr{lt.Comod}_{B\mc{O}}). $$
\end{thm}

\begin{rmk}
Ching expects that one should also get an equivalence of homotopy categories for left modules, but presently do not know how to prove this.
\end{rmk}

\begin{rmk}
In both the case of left and right modules, the bar construction induces equivalences of derived mapping spaces
$$ \ul{\mr{lt./rt.Mod}}_{\mc{O}}(\mc{M}, \mc{N}) \xrightarrow{\simeq} \ul{\mr{lt./rt.Comod}}_{B\mc{O}}(B\mc{M}, B\mc{N}). $$
\end{rmk}

The reader may be startled that the Koszul duality in Theorem~\ref{thm:Koszul} applies to the full categories of modules, and not some suitable subcategory, and makes no mention of ``divided power structures'' (as was the case of the instances of Koszul duality of Section~\ref{sec:Koszul}).  It seems that one should rather think of Theorem~\ref{thm:Koszul} as an extension of Koszul duality for (co)operads, rather than Koszul duality for (co)algebras over (co)operads.  Indeed, regarding an $\mc{O}$-algebra structure on $A$ as a left $\mc{O}$-module structure on $A$, Theorem~\ref{thm:Koszul} does not apply, as the symmetric sequence $(A, \ast, \ast, \cdots)$ does not have trivial $0$th spectrum.  Theorem~\ref{thm:Koszul} (with dualizability hypotheses on $\mc{O}$) does encode an $\mc{O}$-algebra structure on $A$ in a $(B\mc{O})^\vee$-comodule structure on $CA^{\wedge_R *}$, but the latter does not translate into anything like a $B\mc{O}$-coalgebra structure.

\begin{rmk}
Ching does have a \emph{different} Koszul duality Quillen adjunction
\begin{equation}\label{eq:KoszulII}
Q: \mr{lt./rt.Comod}_{B\mc{O}} \leftrightarrows \mr{lt./rt.Mod}_{\mc{O}} : \mr{Prim} 
\end{equation}
which \emph{does not} in general give an equivalence of homotopy categories, but which \emph{does} restrict (in the case of right modules) to give the usual Koszul duality between $(B\mc{O})^\vee$-algebras and $\mc{O}^\vee$-coalgebras.  The monad and comonad of this adjunction encode divided power module and comodule structures, which extend the previously established notions of divided power structures for algebras and coalgebras.
\end{rmk}

\subsection*{The fake Taylor tower}

In \cite{AroneChingchain}, Arone and Ching establish that the derivatives of a functor
$$ F : \mc{C} \rightarrow \mc{D} $$
have the structure of a $\partial_*\Id_\mc{D}$-$\partial_*\Id_\mc{C}$-bimodule.
Note that in the case where either $\mc{C}$ or $\mc{D}$ is the category $\Sp$ of spectra, $\partial_* \Id_\Sp = 1$, and a left or right $\partial_* \Id_\Sp$-module structure amounts to no additional structure.

A key tool, introduced in \cite{AroneChingchain} is the notion of the \emph{fake Taylor tower} of the functor $F$.  The fake Taylor tower is the closest approximation to the Goodwillie tower which can be formed using only the bimodule structure of $\partial_* F$, and is defined as follows. 

For $X \in \mc{C}$, let $R_X$ denote the corepresentable functor
$$ R_X: \mc{C} \rightarrow \mc{D} $$
given by
$$ R_X(Z) = [\Sigma^\infty]\ul{\mc{C}}(X,Z) $$
(where the $\Sigma^\infty$ in the above formula is only used if $\mc{D} = \Sp$).
Then the fake Taylor tower $\{ P_n^{\mit{fake}}F \}$ is the tower of functors under $F$ given by (in the case where $X$ is finite\footnote{For $X$ infinite, one must regard $R_X$ as a pro-functor.})
$$ P^{\mit{fake}}_nF(X) := {}_{\partial_*\Id_{\mc{D}}}\ul{\mr{Bimod}}_{\partial_* \Id_{\mc{C}}}(\partial_* R_X, \tau_n\partial_*F). $$
Here, for a symmetric sequence $\mc{A}$, we are letting $\tau_n \mc{A}$ denote its $n$th truncation
\begin{equation}\label{eq:truncation} 
\tau_n\mc{A}_k := \begin{cases}
\mc{A}_k, & k \le n, \\
\ast, & k > n. 
\end{cases}
\end{equation}
With the hypothesis that all symmetric sequences have trivial $0$th term, it is easy to see that operad and module structures on $\mc{A}$ induce corresponding structures on $\tau_n\mc{A}$.

The layers of the fake Taylor tower given by the fibers
$$ D^\mit{fake}_nF \rightarrow P_n^\mit{fake}F \rightarrow P_{n-1}^\mit{fake}F $$
take the form
$$ D_n^\mit{fake}F(X) \simeq \Omega^\infty_{\mc{D}} \left( \partial_nF \wedge \Sigma^\infty_{\mc{C}} X^n\right)^{h\Sigma_n}. $$

The following theorem is essentially proven in \cite{AroneChingchain}.

\begin{thm}[Arone-Ching]\label{thm:ACfake}
There is a natural transformation of towers 
$$ \{ P_nF \} \rightarrow \{ P^\mit{fake}_n F \} $$
such that the induced map on fibers is given by the norm map
$$ N: \Omega^\infty_{\mc{D}} \left( \partial_nF \wedge \Sigma^\infty_{\mc{C}} X^n\right)_{h\Sigma_n} \rightarrow \Omega^\infty_{\mc{D}} \left( \partial_nF \wedge \Sigma^\infty_{\mc{C}} X^n\right)^{h\Sigma_n}.$$
\end{thm}

Thus, in general, the map from the Goodwillie tower to the fake Taylor tower is not an equivalence, and the difference is measured by the Tate spectra
$$ \Omega^\infty_{\mc{D}} \left( \partial_nF \wedge \Sigma^\infty_{\mc{C}} X^n\right)^{t\Sigma_n}. $$

Although we do not need it for what follows, we pause to mention that Arone and Ching have a refinement of this theory which recovers the Goodwillie tower from descent data on the derivatives.
Observe that the fake Taylor tower only depends on the bimodule $\partial_*F$.  The following is proven in \cite{AroneChing}.

\begin{thm}[Arone-Ching]
The limit of the fake Taylor tower is right adjoint to the derivatives functor:
$$ \partial_* : \mr{Funct}(\mc{C},\mc{D}) \leftrightarrows {}_{\partial_*\Id_{\mc{D}}}{\mr{Bimod}}_{\partial_* \Id_{\mc{C}}}: P_\infty^{\mit{fake}}. $$
\end{thm}

In particular, one can now employ the comonadic descent theory of Section~\ref{sec:general} to regard the derivatives as taking values in $\partial_* \circ P_\infty^{\mit{fake}}$-comodules.

\begin{thm}[Arone-Ching \cite{AroneChing}]\label{thm:AC}
The Goodwillie tower of a functor $F$ can be recovered using the comonadic cobar construction
$$ P_nF \simeq C(P^{\mit{fake}}_\infty, \partial_* \circ P^{\mit{fake}}_\infty, \tau_n \partial_* F). $$
\end{thm}

In the case of functors from spectra to spectra, this theorem reduces to McCarthy's classification of polynomial functors \cite{McCarthy}.

\subsection*{Application to the Bousfield-Kuhn functor}

We now summarize Arone and Ching's approach to Theorem~\ref{thm:main2}.  Actually, their method proves something stronger, as it applies to the functor $\Phi_n$ instead of $\Phi_{K(n)}$.  Call a space \emph{$\Phi_n$-good} if the map 
$$ \Phi_n X \rightarrow \holim_k \Phi_n P_k \Id_{\Top_*} (X) $$
is an equivalence.

\begin{thm}[Arone-Ching]\label{thm:AroneChingvn}
For all finite $X$, the comparison map
$$ c_X: P_\infty\Phi_n(X) \rightarrow \TAQ_{S_{T(n)}}(S_{T(n)}^X) $$
is an equivalence.  Thus for all finite $\Phi_n$-good spaces, the comparison map gives an equivalence
$$ c_X: \Phi_n(X) \xrightarrow{\simeq} \TAQ_{S_{T(n)}}(S_{T(n)}^X). $$
\end{thm}

\begin{proof}
The basic strategy is to analyze the fake taylor tower of the functor
$$ \Phi_n: \Top_* \rightarrow \Sp_{T(n)}. $$
The argument used in Lemma~\ref{lem:Phitower} applies equally well to $\Phi_n$, and it follows that we have
$$ \partial_*\Phi_n \simeq s^{-1}\Lie_{T(n)} $$
with right $\partial_* \Id = s^{-1}\Lie$ structure given by localization of the right action of this operad on itself.
By Theorem~\ref{thm:tate}, the map
\begin{align*}
D_k\Phi_n(X) & = \left\lbrack \left((s^{-1} \Lie_k)_{T(n)} \wedge X^k\right)_{h\Sigma_k} \right\rbrack_{T(n)} \\
& \xrightarrow{N}
\left((s^{-1} \Lie_k)_{T(n)} \wedge X^k\right)^{h\Sigma_k} \\
& = D^{\mit{fake}}_k\Phi_n(X)
\end{align*}
of Theorem~\ref{thm:ACfake} is an equivalence.  Thus in the $T(n)$-local context, the fake Taylor tower agrees with the Goodwillie tower. 
Using \cite[Lemma 6.14]{AroneChing} and Theorem~\ref{thm:Koszul}, we have
\begin{align*}
P_\infty\Phi_n(X) 
& \simeq \ul{\mr{rt.\Mod}}_{s^{-1}\Lie}(\partial_* R_X, s^{-1}\Lie_{T(n)}) \\
& \simeq \ul{\mr{rt.\Mod}}_{s^{-1}\Lie}(B(\Sigma^\infty X^{\wedge *}, \Comm, 1)^\vee, s^{-1}\Lie_{T(n)})  \\
& \simeq \ul{\mr{rt.\Mod}}_{s^{-1}\Lie}\left(C(S_{T(n)}^{X^{\wedge *}}, \Comm_{T(n)}^\vee, 1), C(1,\Comm^\vee_{T(n)},1)\right)  \\
& \simeq \ul{\mr{rt.Comod}}_{\Comm^\vee}(S_{T(n)}^{X^{\wedge *}}, 1_{S_{T(n)}})  \\
& \simeq \Alg_{\Comm}(S_{T(n)}^X, \mr{triv} S_{T(n)}) \\
& \simeq \TAQ_{S_{T(n)}}(S_{T(n)}^X). \\
\end{align*}
\end{proof}

\section{The Heuts approach}\label{sec:Heuts}

The approach of Arone and Ching described in the last section arose from a classification theory of Goodwillie towers.
In this section we describe Heuts' general theoretical framework, which arises from classifying unstable homotopy theories with a fixed stablization \cite{Heuts1}. Our goal is simply to give enough of the idea of the theory to sketch Heuts' proof of Theorem~\ref{thm:main2}.  
We refer the reader to the source material for a proper and more rigorous treatment. 

Like the approach of Arone-Ching, Heuts' proof is more conceptual than ours, and his results have the potential to be slightly more general that Theorem~\ref{thm:AroneChingvn}, in that they seem to indicate that by modifying the comparison map to have target derived primitives of a coalgebra, the comparison map $c_X$ may be an equivalence for \emph{all} $\Phi_n$-good spaces (not just finite spaces --- see Question~\ref{ques:phigood2} and Remark~\ref{rmk:phigood2}).  

Unlike the previous sections, where we worked in a setting of actual categories with weak equivalences, in this section we work in the setting of $\infty$-categories.
For the purposes of this section, $\mc{C}$ will always denote an arbitrary pointed compactly generated $\infty$-category.  

\subsection*{$\infty$-operads and cross-effects}

The adjunction
$$ \Sigma_{\mc{C}}^\infty: \mc{C} \leftrightarrows \Sp(\mc{C}):  \Omega^\infty_{\mc{C}} $$
gives rise to a comonad $\Sigma^\infty_\mc{C} \Omega^\infty_\mc{C}$ on $\Sp(\mc{C})$.  Lurie \cite{Lurie} observes that the multilinearized cross effects
$$ \otimes_{\mc{C}}^n := \mr{cr}^{lin}_n(\Sigma^\infty_{\mc{C}} \Omega^\infty_{\mc{C}}): \Sp(\mc{C})^n \rightarrow \Sp(\mc{C}) $$
get an additional piece of algebraic structure: they corepresent a symmetric multicategory structure on $\Sp(\mc{C})$
in the sense that the mapping spaces 
$$ \ul{\Sp(\mc{C})}\left(\otimes^n_{\mc{C}}(Y_1, \ldots, Y_n), Y\right) $$
endow $\Sp(\mc{C})$ with the structure of a symmetric multicategory enriched in spaces.

If $\Sp(\mc{C}) \simeq \Sp$, then (as discussed in the beginning of Section~\ref{sec:proof}) we have
$$ 
{\otimes}^n_{\mc{C}}(Y_1, \ldots, Y_n) \simeq  \partial_n(\Sigma^\infty_\mc{C} \Omega^\infty_{\mc{C}}) \wedge Y_1 \wedge \cdots \wedge Y_n. $$
Saying that the cross-effects $\otimes^n_{\mc{C}}$ corepresent a symmetric multicategory is equivalent to saying that the derivatives $\partial_*(\Sigma^\infty_{\mc{C}} \Omega^\infty_{\mc{C}})$ form a cooperad.  In this context, this fact was first observed by Arone and Ching \cite{AroneChingchain}, who proved that the derivatives of any comonad on $\Sp$ form a cooperad.

\begin{rmk}
In the language of Lurie, $(\Sp(\mc{C}), \otimes^*_{\mc{C}})$ forms a \emph{stable $\infty$-operad}.  This terminology comes from the fact that a symmetric multicategory is the same thing as a (colored) operad.  We will deliberately avoid this terminology in our treatment, as it may seem somewhat confusing that a stable $\infty$-operad on $\Sp$ is encoded by a \emph{co}operad in $\Sp$.
\end{rmk}

The linearizations of the diagonals in $\mc{C}$
$$ \Delta^n: X \rightarrow X^{\times n} $$
gives rise to $\Sigma_n$-equivariant maps
$$ \Delta^n: \Sigma^\infty_{\mc{C}} X \rightarrow 
\otimes^n_{\mc{C}}(\Sigma^\infty_{\mc{C}} X, \cdots, \Sigma^\infty_\mc{C}X) =:
(\Sigma^\infty_{\mc{C}} X)^{\otimes_{\mc{C}} n} $$
which yield maps
$$ \Delta^n: \Sigma^\infty_{\mc{C}} X \rightarrow 
\left( (\Sigma^\infty_{\mc{C}} X)^{\otimes_{\mc{C}} n} \right)^{h\Sigma_n}. $$
Composing out to the Tate spectrum gives maps
\begin{equation}\label{eq:tatediags}
\delta^n_\mc{C}: \Sigma^\infty_{\mc{C}} X \rightarrow 
\left( (\Sigma^\infty_{\mc{C}} X)^{\otimes_{\mc{C}} n} \right)^{t\Sigma_n}.
\end{equation}
Heuts \cite{Heuts1} refers to these maps as \emph{Tate diagonals}.
In the context of $\mc{C} = \Top_*$, these natural transformations are well studied: their target is closely related to Jones-Wegmann homology (see \cite[II.3]{BMMS}) and the topological Singer construction of Lun\o e-Nielsen-Rognes \cite{LunoeNielsenRognes}.

\subsection*{Polynomial approximations of $\infty$-categories}

Heuts constructs \emph{polynomial approximations} $P_n\mc{C}$: these are $\infty$-categories equipped with adjunctions
$$ \Sigma^\infty_{\mc{C},n}: \mc{C} \leftrightarrows P_n\mc{C}: \Omega^\infty_{\mc{C},n} $$
so that
$$ P_n\Id_\mc{C}(X) \simeq \Omega^\infty_{\mc{C},n} \Sigma^\infty_{\mc{C},n} X. $$
The $\infty$-categories $P_n\mc{C}$ are determined by universal properties which we will not specify here.  We do point out that the identity functor $\Id_{P_n\mc{C}}$ is $n$-excisive.  We have $P_1\mc{C} \simeq \Sp(\mc{C})$.  For $n \le m$ we have 
$$ P_nP_m \mc{C} \simeq P_n\mc{C} $$
and therefore we get a tower
$$
\xymatrix@C+1em@R+1em{
\mc{C} \ar[d]_{\Sigma^\infty_{\mc{C},1}} \ar[dr]^{\Sigma^\infty_{\mc{C},2}} \\
P_1\mc{C} & P_2 \mc{C} \ar[l]^{\Sigma^\infty_{P_2 \mc{C}, 1}} & \cdots \ar[l]^{\Sigma^\infty_{P_3\mc{C},2}}  
}$$
We shall say that an object $X$ of $\mc{C}$ is \emph{convergent} if the Goodwillie tower of $\Id_\mc{C}$ converges at $X$.
Heuts proves that the induced functor 
$$ \mc{C} \rightarrow P_\infty\mc{C} := \holim_n P_n \mc{C} $$
restricts to a full and faithful embedding on the full $\infty$-subcategory 
$\mc{C}^{\mr{conv}}$ of convergent objects.

Let $\mc{C}^{\text{n-conv}}$ denote the full $\infty$-subcategory of $\mc{C}$ consisting of objects for which the map 
$$ X \rightarrow P_n\Id_{\mc{C}}(X) $$
is an equivalence.  Then we have

\begin{lem}\label{lem:nconv}
The functor
$$ \Sigma^\infty_{\mc{C},n}: \mc{C}^{\text{\rm n-conv}} \rightarrow P_n\mc{C} $$
is fully faithful.
\end{lem}

\begin{proof}
We have for $X$ and $Y$ in $\mc{C}^{\text{n-conv}}$:
\begin{align*}
\ul{\mc{C}}(X,Y) 
& \simeq \ul{\mc{C}}(X,\Omega^\infty_{\mc{C},n} \Sigma^\infty_{\mc{C},n}Y)
\\
& \simeq \ul{P_n \mc{C}}(\Sigma^\infty_{\mc{C},n} X,\Sigma^\infty_{\mc{C},n} Y). 
\end{align*}
\end{proof}

The natural transformations
$$ \Sigma^\infty_\mc{C} \Omega^\infty_{\mc{C}} \simeq \Sigma^\infty_{P_n\mc{C}} \Sigma^\infty_{\mc{C},n} \Omega^\infty_{\mc{C},n} \Omega^\infty_{P_n \mc{C}} \rightarrow \Sigma^\infty_{P_n\mc{C}} \Omega^\infty_{P_n\mc{C}} $$
induce natural transformations of cross-effects
$$ \otimes^k_\mc{C} \rightarrow \otimes^k_{P_n\mc{C}}. $$
For $k \le n$ these natural transformations are equivalences.

As the source and target of the Tate diagonals (\ref{eq:tatediags}) are $(n-1)$-excisive functors of $X$ (see \cite{KuhnTate}), the Tate diagonals extend to give natural transformations of functors $P_{n-1}\mc{C} \rightarrow \Sp(\mc{C})$:
$$ \delta^n_\mc{C}: \Sigma^\infty_{P_{n-1}\mc{C}} X \rightarrow 
\left( (\Sigma^\infty_{P_{n-1}\mc{C}} X)^{\otimes_{\mc{C}} n} \right)^{t\Sigma_n}. $$
We emphasize that, as the notation suggests, the Tate diagonals $\{\delta^n_\mc{C}\}_n$ depend not only on the functors $\otimes^*_\mc{C}$ on $\Sp(\mc{C})$, but also on the unstable category $\mc{C}$ itself.

\subsection*{A spectral algebra model for $P_n\mc{C}$}

Heuts gives a model for $P_n\mc{C}$ as a certain category of coalgebras in $\Sp(\mc{C})$.  As the theory of homotopy descent of Section~\ref{sec:general} would have us believe, a good candidate spectral algebra model would be to consider $\Sigma^\infty_{P_n\mc{C}}\Omega^\infty_{P_n\mc{C}}$-coalgebras.  We must analyze what it means for $Y \in \Sp(\mc{C})$ to have a coalgebra structure map
$$ Y \rightarrow \Sigma^\infty_{P_n\mc{C}}\Omega^\infty_{P_n\mc{C}} Y. $$
This is closely related to having a structure map
$$ Y \rightarrow P_n(\Sigma^\infty_\mc{C} \Omega^\infty_\mc{C} ) Y. $$

A general theorem of McCarthy \cite{McCarthy}, as formulated by \cite{KuhnTate}\footnote{To be precise, this is established by McCarthy and Kuhn in the case where $\mc{C} = \Top_*$.}, applies to the functor $\Sigma^\infty_\mc{C} \Omega^\infty_\mc{C}$ to give a homotopy pullback
$$
\xymatrix{
P_n(\Sigma^\infty_{\mc{C}}\Omega^\infty_\mc{C})(Y) \ar[r] \ar[d] 
& (Y^{\otimes_{\mc{C}} n})^{h\Sigma_n} \ar[d] \\
P_{n-1}(\Sigma^\infty_{\mc{C}}\Omega^\infty_\mc{C})(Y) \ar[r] 
& (Y^{\otimes_{\mc{C}} n})^{t\Sigma_n}
}
$$
Thus inductively a $\Sigma^\infty_{P_n\mc{C}}\Omega^\infty_{P_n\mc{C}}$-coalgebra is determined by the data of a map
$$ Y \rightarrow P_{n-1}(\Sigma^\infty_\mc{C}\Omega^\infty_\mc{C})(Y) $$
and a lifting\footnote{This is something the first author learned from Arone.}
$$
\xymatrix{
&& (Y^{\otimes_{\mc{C}} n})^{h\Sigma_n} \ar[d] \\
Y \ar[r] \ar@/^1pc/@{.>}[rru] & P_{n-1}(\Sigma^\infty_{\mc{C}}\Omega^\infty_\mc{C})(Y) \ar[r] 
& (Y^{\otimes_{\mc{C}} n})^{t\Sigma_n}
}
$$
The bottom composite agrees with the Tate diagonal $\delta^n_{\mc{C}}$ for $Y = \Sigma^\infty_{\mc{C},n-1} X$.

We will refer to these coalgebras as \emph{Tate-compatible $\otimes^{\le n}_\mc{C}$-coalgebras}, and denote the $\infty$-category of such
$$ \mr{Tate}\Coalg_{\otimes^{\le n}_{\mc{C}}}. $$
Roughly speaking, a Tate-compatible $\otimes^{\le n}_\mc{C}$-coalgebra is an object $Y \in \Sp(\mc{C})$ equipped with inductively defined structure consisting of coaction maps
$$ \Delta^k: Y \rightarrow (Y^{\otimes_\mc{C} k})^{h\Sigma_k} $$ 
for $k \le n$, and homotopies $H_k$ making the following diagrams homotopy commute
$$
\xymatrix@C+1em{
& (Y^{\otimes_\mc{C} n})^{h\Sigma_k} \ar[d]
\\
Y \ar[ur]^{\Delta^k} \ar[r]_-{\delta^k_{\mc{C}}} & (Y^{\otimes_\mc{C} k})^{t\Sigma_k}
}
$$
The coaction maps $\Delta^k$ and the homotopies $H_k$ are required to satisfy compatibility conditions which we will not (and likely cannot!) explicitly specify.\footnote{Heuts is able to circumvent the need to explicitly spell out these compatibility conditions by defining the $\infty$-categories $\mr{Tate}\Coalg_{\otimes^{\le n}_{\mc{C}}}$  via an inductive sequence of fibrations of $\infty$-categories.}
The maps $\Delta^k$ and homotopies $H_k$ for $k \le n$ then induce the $(n+1)$st Tate diagonal
$$ \delta^{n+1}_{\mc{C}}: Y \rightarrow (Y^{\otimes_\mc{C} {n+1}})^{t\Sigma_{n+1}}
$$
and the process continues.  Note that the Tate diagonal $\delta^{n+1}_{\mc{C}}$ depends not only on the structure maps $\Delta^k$ and $H_k$ for $k \le n$, but also the unstable category $\mc{C}$ itself (more precisely, it depends only on the polynomial approximation $P_n\mc{C}$).

\begin{thm}[Heuts]\label{thm:Tatecoalgebra}
There is an equivalence of $\infty$-categories
$$ P_n\mc{C} \simeq \mr{Tate}\Coalg_{\otimes^{\le n}_\mc{C}}. $$
\end{thm}

\begin{ques}
In the case where $F = \Id$, how is Arone-Ching's reconstruction theorem (Theorem~\ref{thm:AC}) related to the framework of Heuts?
\end{ques}

\begin{rmk}
In \cite{Heuts1}, Heuts also considers the question: what data on the stable $\infty$-category $\Sp(\mc{C})$ determines the tower of unstable categories $\{ P_n\mc{C} \}$?  As should be heuristically clear from Theorem~\ref{thm:Tatecoalgebra}, Heuts proves the tower is determined by the cross-effects $\{ \otimes^n_{\mc{C}} \}$ and the Tate diagonals $\{\delta^n_{\mc{C}}\}$. In particular, given a stable $\infty$-category $\mc{D}$, a tower of polynomial approximations of an unstable theory is determined by specifying a sequence of symmetric multilinear functors 
$$
\otimes^n : \mc{D}^n \rightarrow \mc{D}
$$ 
which corepresent a symmetric multicategory structure on $\mc{D}$, as well as a sequence of inductively defined (and suitably compatible) Tate diagonals
$$ \delta^n : \Sigma^\infty_{P_{n-1}\mc{C}} X \rightarrow (\Sigma^\infty_{P_{n-1}\mc{C}} X^{\otimes n})^{t\Sigma_n}. $$
\end{rmk}

\subsection*{Koszul duality, yet again}

Let $R$ be a commutative ring spectrum, and let $\mc{O}$ be a reduced operad in $\Mod_R$.  Following \cite{Heuts1}, we run the general theory in the case $\mc{C} = \Alg_\mc{O}$.  The cooperads representing the symmetric multilinear functors $\otimes^*_{\Alg_\mc{O}}$ on $\Sp(\Alg_\mc{O}) \simeq \Mod_R$ are determined by the following 

\begin{thm}[Francis-Gaitsgory {\cite[Lem.~3.3.4]{FrancisGaitsgory}}]
There is an equivalence of cooperads\footnote{This relies on the treatment of Koszul duality of monoids in \cite{Lurie}.  In Lurie's $\infty$-categorical treatment, the coalgebra structure on $B\mc{O}$ making this theorem true is only coherently homotopy associative.  Presumably it can be strictified to an actual point-set level operad structure on a model of $B\mc{O}$, but the authors are not knowledgeable enough to know the feasibility of this, nor do they know if this cooperad structure is equivalent to that of Ching \cite{Ching}.}
$$ \partial_*(\Sigma^\infty_{\Alg_{\mc{O}}} \Omega^\infty_{\Alg_\mc{O}}) \simeq B\mc{O}. $$
\end{thm}

Therefore a $\otimes^*_{\Alg_\mc{O}}$-coalgebra $A$ is simply a $B\mc{O}$-coalgebra.
The Tate diagonals on $\Mod_R$ turn out to be null in this case, so a Tate compatible structure on a $B\mc{O}$-coalgebra $A$ is a compatible choice of liftings of the coaction maps
$$
\xymatrix{
& (B\mc{O}_i \wedge_R A^i)_{h\Sigma_i} \ar[d] \\
A \ar[r] \ar@{.>}[ur] & (B\mc{O}_i \wedge_R A^i)^{h\Sigma_i}
}
$$
Thus a Tate compatible structure is the same thing as a divided power structure (or perhaps one can take this as a definition of a divided power structure).  We shall denote the $\infty$-category of such (with structure maps as above for $i \le n$) by $\mr{d.p.}\Coalg_{B\mc{O}^{\le n}}$.

\begin{thm}[Heuts]\label{thm:halfKoszul}
There are equivalences of $\infty$-categories
$$ P_n\Alg_\mc{O} \simeq \mr{d.p.}\Coalg_{B\mc{O}^{\le n}}. $$
\end{thm}

Heuts recovers the following weak Koszul duality result.

\begin{cor}[Heuts]\label{cor:halfKoszul}
There is a fully faithful embedding
$$ \TAQ^{\mc{O}}: \Alg_\mc{O}^{\mr{conv}} \hookrightarrow \holim_n \mr{d.p.}\Coalg_{B\mc{O}^{\le n}}. $$
\end{cor}

To determine the convergent objects of $\Id_{\Alg_\mc{O}}$, it is helpful to know the structure of this Goodwillie tower.  The following result was suggested by Harper and Hess \cite{HarperHess}, was proven in the case of the commutative operad by Kuhn \cite{KuhnAdv}, and was proven by Pereira \cite{Pereira2}.

\begin{thm}[Pereira]
The Goodwillie tower of $\Id_{\Alg_\mc{O}}$ is given by
$$ P_n\Id_{\Alg_\mc{O}}(A) = B(\mc{F}_{\tau_n\mc{O}}, \mc{F}_\mc{O}, A). $$
Here $\tau_n\mc{O}$ denotes the truncation (\ref{eq:truncation}).
\end{thm}

In particular, connectivity estimates of Harper and Hess \cite{HarperHess} imply that if $R$ and $\mc{O}$ are connective, and $A$ is connected, then $A$ is convergent.  Thus Corollary~\ref{cor:halfKoszul} recovers half of Theorem~\ref{thm:Koszulduality}. Another important case are operads for which $\mc{O} = \tau_{n}\mc{O}$.  Then every $\mc{O}$-algebra is convergent, and Corollary~\ref{cor:halfKoszul} recovers a theorem of Cohn.

\subsection*{Application to unstable $v_n$-periodic homotopy}

To recover and generalize Theorem~\ref{thm:main2}, Heuts applies his general framework to the unstable $v_n$-periodic homotopy category.  Unfortunately, the $\infty$-category modeling $M_n^f\Top_*$ of Section~\ref{sec:vn} seems to fail to be compactly generated.  To rectify this, Heuts works with a slightly different $\infty$-category, which we will denote $v_n^{-1}\Top_*$.  This is the full $\infty$-subcategory of $L_n^f\Top_*$ consisting of colimits of finite $(d_n-1)$-connected type $n$ complexes. The categories $v_n^{-1}\Top_*$ and $M^f_n\Top_*$ are very closely related. The Bousfield-Kuhn functor factors as
\begin{equation}\label{eq:factor}
\xymatrix{ 
\Top_* \ar[rr]^{\Phi_n} \ar[dr] && \Sp_{T(n)} \\
& v_n^{-1}\Top_* \ar@{.>}[ur]_{\Phi_n'} 
}
\end{equation}
and detects the equivalences in $v_n^{-1}\Top_*$.

We have $\Sp(v_n^{-1}\Top_*) \simeq \Sp_{T(n)}$.  The multilinear cross-effects are given by the commutative cooperad:
$$ \partial_*(\Sigma^\infty\Omega^\infty) \simeq \Comm^\vee. $$
In this context Theorem~\ref{thm:tate} implies that the Tate diagonals are trivial, and Tate compatible commutative coalgebras are the same thing as commutative coalgebras.  Heuts deduces (using Theorems~\ref{thm:Tatecoalgebra} and \ref{thm:halfKoszul}):

\begin{thm}[Heuts]\label{thm:vn}
There are equivalences of $\infty$-categories
$$ P_k(v_n^{-1}\Top_*) \simeq \Coalg_{\Comm^{\le k}}(\Sp_{T(n)}) \simeq P_k (\Alg_{s^{-1}\Lie}(\Sp_{T(n)})). $$
\end{thm}

In a sense made precise in the corollary below, this gives two spectral algebra models of $\mc{C}$.

\begin{cor}
There are fully faithful embeddings of $\infty$-categories
\begin{align*}
(v_n^{-1}\Top_*)^{\mr{conv}} & \hookrightarrow \holim_k \Coalg_{(\Comm^\vee)^{\le k}}(\Sp_{T(n)}), \\
(v_n^{-1}\Top_*)^{\mr{conv}} & \hookrightarrow P_\infty \Alg_{s^{-1}\Lie}(\Sp_{T(n)}).
\end{align*}
\end{cor}

We can be explicit about the functors giving these spectral algebra models. 
In general there is an adjunction
$$ \mr{triv} : \Mod_R \leftrightarrows \Coalg_{\Comm^\vee_R}: \mr{Prim}  $$
where $\mr{triv} \: Y$ is the coalgebra with trivial coproduct, and $\mr{Prim}(A)$ is the \emph{derived primitives} of a coalgebra $A$, given by the comonadic cobar construction:
$$ \mr{Prim}(A) := C(\Id, \mc{F}_{\mr{Comm}^\vee_R}, A). $$
For $A$ a $\Comm_R$-algebra finite as an $R$-module, we have
\begin{equation}\label{eq:TAQPrim}
\TAQ_R(A) \simeq \mr{Prim} (A^\vee).
\end{equation}
Ching's work endows $\mr{Prim}(A)$ with the structure of an $s^{-1}\Lie$-algebra.  

The functors of Theorem~\ref{thm:vn} are induced from the functors
$$ v^{-1}_n\Top_* \xrightarrow{(\Sigma^\infty -)_{T(n)}} \Coalg_{\mr{Comm}^\vee}(\Sp_{T(n)}) \xrightarrow{\Prim} \Alg_{s^{-1}\Lie}(\Sp_{T(n)}). $$
An argument following the same lines as Section~\ref{sec:comparison} gives a refined comparison map
$$ \td{c}_X: \Phi_n(X) \rightarrow \Prim (\Sigma^\infty X)_{T(n)}. $$
Under the equivalence (\ref{eq:TAQPrim}), this agrees with the comparison map $c_{X}$ for $X$ finite, and for such $X$ gives $c_X^{K(n)}$ after $K(n)$-localization.  From Theorem~\ref{thm:vn}, Heuts deduces that for a space $X$, the comparison map refines to an equivalence of towers
\begin{equation}\label{eq:twotowersII}
 \td{c}_X : \Phi_n P_k\Id_{\Top_*} X \xrightarrow{\simeq} \mr{Prim} \: 
\Omega^\infty_{\mr{\Coalg}, k} \Sigma^\infty_{\mr{\Coalg},k} (\Sigma^\infty X)_{T(n)}.
\end{equation}
Using Theorem~\ref{thm:AroneMahowald}, Heuts obtains the following refinement of Theorem~\ref{thm:main}.

\begin{cor}[Heuts]
The comparison map $\td{c}_X$ is an equivalence for $X$ a sphere. 
\end{cor}

\begin{ques}\label{ques:phigood}
What is the relationship between the $\infty$-subcategory $(v_n^{-1}\Top_*)^{\mr{conv}} \subseteq v_n^{-1}\Top_*$ and the $\infty$-subcategory consisting of the images of $\Phi_n$-good spaces?
\end{ques}

\begin{rmk}
If we knew that the functor $\Phi'_{n}$ of (\ref{eq:factor}) preserved homotopy limits, then it is fairly easy to check (using the fact that $\Phi'_n$ detects equivalences) that the two $\infty$-subcategories of Question~\ref{ques:phigood} would in fact coincide.  As already remarked in Section~\ref{sec:vn}, $\Phi_n$ also factors through a related functor
$$ \Phi_{n}'': M_n^f \Top_* \rightarrow \Sp_{T(n)}. $$
Bousfield produces a left adjoint for $\Phi''_n$ in \cite{Bousfield}, and it therefore follows that $\Phi_n''$ commutes with homotopy limits.
\end{rmk}

It would seem that for $X$ an infinite CW complex, the coalgebra $(\Sigma^\infty X)_{T(n)}$ is a more appropriate model for the unstable $v_n$-periodic homotopy type $X$ than the algebra $S_{T(n)}^X$.  To this end we ask the following

\begin{ques}\label{ques:phigood2}
Is $\td{c}_X$ an equivalence for \emph{all} $\Phi_n$-good spaces $X$?
\end{ques}

\begin{rmk}\label{rmk:phigood2}
We expect the answer to Question~\ref{ques:phigood2} should be ``yes'', as the tower which is the target of (\ref{eq:twotowersII}) should be an analog for primitives of the Kuhn filtration, and hence should converge without hypotheses.
\end{rmk}

\bibliographystyle{amsalpha}
\nocite{*}
\bibliography{BKTAQsurvey3}

\providecommand{\bysame}{\leavevmode\hbox to3em{\hrulefill}\thinspace}
\providecommand{\MR}{\relax\ifhmode\unskip\space\fi MR }
\providecommand{\MRhref}[2]{%
  \href{http://www.ams.org/mathscinet-getitem?mr=#1}{#2}
}
\providecommand{\href}[2]{#2}
\begin{thebibliography}{EKMM97}

\bibitem[AC]{AroneChingvn}
Gregory Arone and Michael Ching, \emph{Localized {T}aylor towers}, In
  preperation.

\bibitem[AC11]{AroneChingchain}
Greg Arone and Michael Ching, \emph{Operads and chain rules for the calculus of
  functors}, Ast\'erisque (2011), no.~338, vi+158.

\bibitem[AC15]{AroneChing}
Gregory Arone and Michael Ching, \emph{A classification of {T}aylor towers of
  functors of spaces and spectra}, Adv. Math. \textbf{272} (2015), 471--552.

\bibitem[AC16]{AroneChingcross}
\bysame, \emph{Cross-effects and the classification of {T}aylor towers}, Geom.
  Topol. \textbf{20} (2016), no.~3, 1445--1537.

\bibitem[AF15]{AyalaFrancis}
David Ayala and John Francis, \emph{Zero-pointed manifolds}, arXiv:1409.2857,
  2015.

\bibitem[AK98]{AroneKankaanrintaSnaith}
Greg Arone and Marja Kankaanrinta, \emph{A functorial model for iterated
  {S}naith splitting with applications to calculus of functors}, Stable and
  unstable homotopy ({T}oronto, {ON}, 1996), Fields Inst. Commun., vol.~19,
  Amer. Math. Soc., Providence, RI, 1998, pp.~1--30.

\bibitem[AK14]{AroneKankaanrinta}
Gregory Arone and Marja Kankaanrinta, \emph{The sphere operad}, Bull. Lond.
  Math. Soc. \textbf{46} (2014), no.~1, 126--132.

\bibitem[AM99]{AroneMahowald}
Greg Arone and Mark Mahowald, \emph{The {G}oodwillie tower of the identity
  functor and the unstable periodic homotopy of spheres}, Invent. Math.
  \textbf{135} (1999), no.~3, 743--788.

\bibitem[Bas99]{Basterra}
M.~Basterra, \emph{Andr\'e-{Q}uillen cohomology of commutative {$S$}-algebras},
  J. Pure Appl. Algebra \textbf{144} (1999), no.~2, 111--143.

\bibitem[BH16a]{BlomquistHarper}
Jacobson Blomquist and John Harper, \emph{An integral chains analog of
  {Q}uillen's rational homotopy theory equivalence}, arXiv:1611.04157, 2016.

\bibitem[BH16b]{BrantnerHeuts}
Lukas Brantner and Gijs Heuts, \emph{The $v_n$-periodic {G}oodwillie tower on
  wedges and cofibers}, arXiv:1612.02694, 2016.

\bibitem[BJM15]{BauerJohnsonMcCarthy}
Kristine Bauer, Brenda Johnson, and Randy McCarthy, \emph{Cross effects and
  calculus in an unbased setting}, Trans. Amer. Math. Soc. \textbf{367} (2015),
  no.~9, 6671--6718, With an appendix by Rosona Eldred.

\bibitem[BM02]{BasterraMcCarthy}
Maria Basterra and Randy McCarthy, \emph{{$\Gamma$}-homology, topological
  {A}ndr\'e-{Q}uillen homology and stabilization}, Topology Appl. \textbf{121}
  (2002), no.~3, 551--566.

\bibitem[BM05]{BasterraMandell}
Maria Basterra and Michael~A. Mandell, \emph{Homology and cohomology of
  {$E_\infty$} ring spectra}, Math. Z. \textbf{249} (2005), no.~4, 903--944.

\bibitem[BMMS86]{BMMS}
R.~R. Bruner, J.~P. May, J.~E. McClure, and M.~Steinberger, \emph{{$H_\infty $}
  ring spectra and their applications}, Lecture Notes in Mathematics, vol.
  1176, Springer-Verlag, Berlin, 1986.

\bibitem[Bou94]{Bousfield94}
A.~K. Bousfield, \emph{Localization and periodicity in unstable homotopy
  theory}, J. Amer. Math. Soc. \textbf{7} (1994), no.~4, 831--873.

\bibitem[Bou99]{Bousfield99}
\bysame, \emph{The {$K$}-theory localizations and {$v_1$}-periodic homotopy
  groups of {$H$}-spaces}, Topology \textbf{38} (1999), no.~6, 1239--1264.

\bibitem[Bou01]{Bousfield}
\bysame, \emph{On the telescopic homotopy theory of spaces}, Trans. Amer. Math.
  Soc. \textbf{353} (2001), no.~6, 2391--2426 (electronic).

\bibitem[Bou05]{Bousfield05}
\bysame, \emph{On the 2-primary {$v_1$}-periodic homotopy groups of spaces},
  Topology \textbf{44} (2005), no.~2, 381--413.

\bibitem[Bou07]{Bousfield07}
\bysame, \emph{On the 2-adic {$K$}-localizations of {$H$}-spaces}, Homology,
  Homotopy Appl. \textbf{9} (2007), no.~1, 331--366.

\bibitem[BR14]{BiedermannRondigs}
Georg Biedermann and Oliver R{\"o}ndigs, \emph{Calculus of functors and model
  categories, {II}}, Algebr. Geom. Topol. \textbf{14} (2014), no.~5,
  2853--2913.

\bibitem[BR15]{BKTAQ}
Mark Behrens and Charles Rezk, \emph{The {B}ousfield-{K}uhn functor and
  topological {A}ndr\'e-{Q}uillen cohomology}, available at {\tt
  www.nd.edu/$\sim$mbehren1/papers}, 2015.

\bibitem[Bra]{Brantner}
Lukas Brantner, \emph{The {L}ubin-{T}ate theory of spectral {L}ie algebras}, In
  preparation.

\bibitem[CH15]{ChingHarper}
Ching and Harper, \emph{Derived koszul duality and {TQ}-homology completion of
  structured ring spectra}, arXiv:1502.06944, 2015.

\bibitem[Chi]{ChingKoszul}
Michael Ching, \emph{Koszul duality for modules and comodules over operads of
  spectra}, In preparation.

\bibitem[Chi05]{Ching}
\bysame, \emph{Bar constructions for topological operads and the {G}oodwillie
  derivatives of the identity}, Geom. Topol. \textbf{9} (2005), 833--933
  (electronic).

\bibitem[CM17]{ClausenMathew}
Dustin Clausen and Akhil Mathew, \emph{A short proof of telescopic {T}ate
  vanishing}, To appear in Proc. Amer. Math. Soc., 2017.

\bibitem[Dav02]{Davis}
Donald~M. Davis, \emph{From representation theory to homotopy groups}, Mem.
  Amer. Math. Soc. \textbf{160} (2002), no.~759, viii+50.

\bibitem[DHKS04]{DHKS}
William~G. Dwyer, Philip~S. Hirschhorn, Daniel~M. Kan, and Jeffrey~H. Smith,
  \emph{Homotopy limit functors on model categories and homotopical
  categories}, Mathematical Surveys and Monographs, vol. 113, American
  Mathematical Society, Providence, RI, 2004.

\bibitem[EKMM97]{EKMM}
A.~D. Elmendorf, I.~Kriz, M.~A. Mandell, and J.~P. May, \emph{Rings, modules,
  and algebras in stable homotopy theory}, Mathematical Surveys and Monographs,
  vol.~47, American Mathematical Society, Providence, RI, 1997, With an
  appendix by M. Cole.

\bibitem[FG12]{FrancisGaitsgory}
John Francis and Dennis Gaitsgory, \emph{Chiral {K}oszul duality}, Selecta
  Math. (N.S.) \textbf{18} (2012), no.~1, 27--87.

\bibitem[Fre00]{Fressedp}
Benoit Fresse, \emph{On the homotopy of simplicial algebras over an operad},
  Trans. Amer. Math. Soc. \textbf{352} (2000), no.~9, 4113--4141.

\bibitem[Fre04]{Fresse}
\bysame, \emph{Koszul duality of operads and homology of partition posets},
  Homotopy theory: relations with algebraic geometry, group cohomology, and
  algebraic {$K$}-theory, Contemp. Math., vol. 346, Amer. Math. Soc.,
  Providence, RI, 2004, pp.~115--215.

\bibitem[Fre10]{French}
Jennifer French, \emph{Derived mapping spaces as models for localizations},
  Ph.D. thesis, M.I.T., 2010.

\bibitem[GH00]{GoerssHopkins}
Paul~G. Goerss and Michael~J. Hopkins, \emph{Andr\'e-{Q}uillen (co)-homology
  for simplicial algebras over simplicial operads}, Une d\'egustation
  topologique [{T}opological morsels]: homotopy theory in the {S}wiss {A}lps
  ({A}rolla, 1999), Contemp. Math., vol. 265, Amer. Math. Soc., Providence, RI,
  2000, pp.~41--85.

\bibitem[GJ]{GetzlerJones}
Getzler and Jones, \emph{Operads, homotopy algebra and iterated integrals for
  double loop spaces}, arXiv:hep-th/9403055.

\bibitem[GK94]{GinzburgKapranov}
Victor Ginzburg and Mikhail Kapranov, \emph{Koszul duality for operads}, Duke
  Math. J. \textbf{76} (1994), no.~1, 203--272.

\bibitem[GM95]{GreenleesMay}
J.~P.~C. Greenlees and J.~P. May, \emph{Generalized {T}ate cohomology}, Mem.
  Amer. Math. Soc. \textbf{113} (1995), no.~543, viii+178.

\bibitem[Goe95]{Goerss}
Paul~G. Goerss, \emph{Simplicial chains over a field and {$p$}-local homotopy
  theory}, Math. Z. \textbf{220} (1995), no.~4, 523--544.

\bibitem[Goo03]{Goodwillie}
Thomas~G. Goodwillie, \emph{Calculus. {III}. {T}aylor series}, Geom. Topol.
  \textbf{7} (2003), 645--711 (electronic).

\bibitem[GS96]{GreenleesSadofsky}
J.~P.~C. Greenlees and Hal Sadofsky, \emph{The {T}ate spectrum of
  {$v_n$}-periodic complex oriented theories}, Math. Z. \textbf{222} (1996),
  no.~3, 391--405.

\bibitem[Har10]{Harper}
John~E. Harper, \emph{Bar constructions and {Q}uillen homology of modules over
  operads}, Algebr. Geom. Topol. \textbf{10} (2010), no.~1, 87--136.

\bibitem[Hes10]{Hess}
Kathryn Hess, \emph{A general framework for homotopic descent and codescent},
  arXiv:1001.1556, 2010.

\bibitem[Heu]{Heuts2}
Gijs Heuts, \emph{Periodicity in unstable homotopy}, in preparation.

\bibitem[Heu16]{Heuts1}
\bysame, \emph{Goodwillie approximations to higher categories},
  arXiv:1510.03304, 2016.

\bibitem[HH13]{HarperHess}
John~E. Harper and Kathryn Hess, \emph{Homotopy completion and topological
  {Q}uillen homology of structured ring spectra}, Geom. Topol. \textbf{17}
  (2013), no.~3, 1325--1416.

\bibitem[Hir03]{Hirschhorn}
Philip~S. Hirschhorn, \emph{Model categories and their localizations},
  Mathematical Surveys and Monographs, vol.~99, American Mathematical Society,
  Providence, RI, 2003.

\bibitem[Hov01]{Hovey}
Mark Hovey, \emph{Spectra and symmetric spectra in general model categories},
  J. Pure Appl. Algebra \textbf{165} (2001), no.~1, 63--127.

\bibitem[HR92]{HopkinsRavenel}
Michael~J. Hopkins and Douglas~C. Ravenel, \emph{Suspension spectra are
  harmonic}, Bol. Soc. Mat. Mexicana (2) \textbf{37} (1992), no.~1-2, 271--279,
  Papers in honor of Jos{\'e} Adem (Spanish).

\bibitem[HS96]{HoveySadofsky}
Mark Hovey and Hal Sadofsky, \emph{Tate cohomology lowers chromatic {B}ousfield
  classes}, Proc. Amer. Math. Soc. \textbf{124} (1996), no.~11, 3579--3585.

\bibitem[HS98]{HopkinsSmith}
Michael~J. Hopkins and Jeffrey~H. Smith, \emph{Nilpotence and stable homotopy
  theory. {II}}, Ann. of Math. (2) \textbf{148} (1998), no.~1, 1--49.

\bibitem[HS14]{HessShipley}
Kathryn Hess and Brooke Shipley, \emph{The homotopy theory of coalgebras over a
  comonad}, Proc. Lond. Math. Soc. (3) \textbf{108} (2014), no.~2, 484--516.

\bibitem[Joh95]{Johnson}
Brenda Johnson, \emph{The derivatives of homotopy theory}, Trans. Amer. Math.
  Soc. \textbf{347} (1995), no.~4, 1295--1321.

\bibitem[Kle05]{Klein}
John~R. Klein, \emph{Moduli of suspension spectra}, Trans. Amer. Math. Soc.
  \textbf{357} (2005), no.~2, 489--507.

\bibitem[KP17]{KuhnPereira}
Nicholas~J. Kuhn and Lu\'\i s~Alexandre Pereira, \emph{Operad bimodules and
  composition products on {A}ndr\'e-{Q}uillen filtrations of algebras}, Algebr.
  Geom. Topol. \textbf{17} (2017), no.~2, 1105--1130.

\bibitem[K{\v{r}}{\'{\i}}93]{Kriz}
Igor K{\v{r}}{\'{\i}}{\v{z}}, \emph{{$p$}-adic homotopy theory}, Topology Appl.
  \textbf{52} (1993), no.~3, 279--308.

\bibitem[Kuh04a]{KuhnMcCord}
Nicholas~J. Kuhn, \emph{The {M}c{C}ord model for the tensor product of a space
  and a commutative ring spectrum}, Categorical decomposition techniques in
  algebraic topology ({I}sle of {S}kye, 2001), Progr. Math., vol. 215,
  Birkh\"auser, Basel, 2004, pp.~213--236.

\bibitem[Kuh04b]{KuhnTate}
\bysame, \emph{Tate cohomology and periodic localization of polynomial
  functors}, Invent. Math. \textbf{157} (2004), no.~2, 345--370.

\bibitem[Kuh06]{KuhnAdv}
\bysame, \emph{Localization of {A}ndr\'e-{Q}uillen-{G}oodwillie towers, and the
  periodic homology of infinite loopspaces}, Adv. Math. \textbf{201} (2006),
  no.~2, 318--378.

\bibitem[Kuh07]{Kuhncalc}
\bysame, \emph{Goodwillie towers and chromatic homotopy: an overview},
  Proceedings of the {N}ishida {F}est ({K}inosaki 2003), Geom. Topol. Monogr.,
  vol.~10, Geom. Topol. Publ., Coventry, 2007, pp.~245--279.

\bibitem[Kuh08]{Kuhn}
\bysame, \emph{A guide to telescopic functors}, Homology, Homotopy Appl.
  \textbf{10} (2008), no.~3, 291--319.

\bibitem[LNR12]{LunoeNielsenRognes}
Sverre Lun{\o}e-Nielsen and John Rognes, \emph{The topological {S}inger
  construction}, Doc. Math. \textbf{17} (2012), 861--909.

\bibitem[Lur16]{Lurie}
Jacob Lurie, \emph{Higher algebra}, available at {\tt
  www.math.harvard.edu/$\sim$lurie/}, 2016.

\bibitem[Mah81]{Mahowaldbo}
Mark Mahowald, \emph{{$b{\rm o}$}-resolutions}, Pacific J. Math. \textbf{92}
  (1981), no.~2, 365--383.

\bibitem[Mah82]{MahowaldImJ}
\bysame, \emph{The image of {$J$} in the {$EHP$} sequence}, Ann. of Math. (2)
  \textbf{116} (1982), no.~1, 65--112.

\bibitem[Man01]{Mandell}
Michael~A. Mandell, \emph{{$E_\infty$} algebras and {$p$}-adic homotopy
  theory}, Topology \textbf{40} (2001), no.~1, 43--94.

\bibitem[Man06]{Mandellcochains}
\bysame, \emph{Cochains and homotopy type}, Publ. Math. Inst. Hautes \'Etudes
  Sci. (2006), no.~103, 213--246.

\bibitem[McC01]{McCarthy}
Randy McCarthy, \emph{Dual calculus for functors to spectra}, Homotopy methods
  in algebraic topology ({B}oulder, {CO}, 1999), Contemp. Math., vol. 271,
  Amer. Math. Soc., Providence, RI, 2001, pp.~183--215.

\bibitem[Med]{Medina}
Anibal Medina, \emph{${E}_\infty$-comodules and topological manifolds}.

\bibitem[Mil81]{MillerASS}
Haynes~R. Miller, \emph{On relations between {A}dams spectral sequences, with
  an application to the stable homotopy of a {M}oore space}, J. Pure Appl.
  Algebra \textbf{20} (1981), no.~3, 287--312.

\bibitem[MRS01]{MRS}
Mark Mahowald, Douglas Ravenel, and Paul Shick, \emph{The triple loop space
  approach to the telescope conjecture}, Homotopy methods in algebraic topology
  ({B}oulder, {CO}, 1999), Contemp. Math., vol. 271, Amer. Math. Soc.,
  Providence, RI, 2001, pp.~217--284.

\bibitem[MS88]{MahowaldShick}
Mark Mahowald and Paul Shick, \emph{Root invariants and periodicity in stable
  homotopy theory}, Bull. London Math. Soc. \textbf{20} (1988), no.~3,
  262--266.

\bibitem[MSS02]{MarklSniderStasheff}
Martin Markl, Steve Shnider, and Jim Stasheff, \emph{Operads in algebra,
  topology and physics}, Mathematical Surveys and Monographs, vol.~96, American
  Mathematical Society, Providence, RI, 2002.

\bibitem[Per13]{Pereira1}
Lu\'is~Alexandre Pereira, \emph{A general context for goodwillie calculus},
  arXiv:1301.2832, 2013.

\bibitem[Per15]{Pereira2}
\bysame, \emph{Goodwillie calculus in algebras over a spectral operad},
  available at {\tt http://www.faculty.virginia.edu/luisalex/research}, 2015.

\bibitem[Per16]{PereiraOperad}
Lu\'\i s~Alexandre Pereira, \emph{Cofibrancy of operadic constructions in
  positive symmetric spectra}, Homology Homotopy Appl. \textbf{18} (2016),
  no.~2, 133--168.

\bibitem[Pet68]{Petrie}
Ted Petrie, \emph{The weakly complex bordism of {L}ie groups}, Ann. of Math.
  (2) \textbf{88} (1968), 371--402.

\bibitem[Pri70]{Priddy}
Stewart~B. Priddy, \emph{Koszul resolutions}, Trans. Amer. Math. Soc.
  \textbf{152} (1970), 39--60.

\bibitem[Qui67]{HA}
Daniel~G. Quillen, \emph{Homotopical algebra}, Lecture Notes in Mathematics,
  No. 43, Springer-Verlag, Berlin-New York, 1967.

\bibitem[Qui69]{Quillen}
Daniel Quillen, \emph{Rational homotopy theory}, Ann. of Math. (2) \textbf{90}
  (1969), 205--295.

\bibitem[Rav84]{Ravenel84}
Douglas~C. Ravenel, \emph{Localization with respect to certain periodic
  homology theories}, Amer. J. Math. \textbf{106} (1984), no.~2, 351--414.

\bibitem[Rez08]{Rezk08}
Charles Rezk, \emph{Power operations for {M}orava {E}-theory of height 2 at the
  prime 2}, arXiv:0812.1320, 2008.

\bibitem[Rez09]{Rezk}
\bysame, \emph{The congruence criterion for power operations in {M}orava
  {$E$}-theory}, Homology, Homotopy Appl. \textbf{11} (2009), no.~2, 327--379.

\bibitem[Rez12a]{RezkMIC}
\bysame, \emph{Modular isogeny complexes}, Algebr. Geom. Topol. \textbf{12}
  (2012), no.~3, 1373--1403.

\bibitem[Rez12b]{RezkKoszul}
\bysame, \emph{Rings of power operations for {M}orava {E}-theories are
  {K}oszul}, arXiv:1204.4831, 2012.

\bibitem[Sch97]{Schwede}
Stefan Schwede, \emph{Spectra in model categories and applications to the
  algebraic cotangent complex}, J. Pure Appl. Algebra \textbf{120} (1997),
  no.~1, 77--104.

\bibitem[SS85]{SchlessingerStasheff}
Michael Schlessinger and James Stasheff, \emph{The {L}ie algebra structure of
  tangent cohomology and deformation theory}, J. Pure Appl. Algebra \textbf{38}
  (1985), no.~2-3, 313--322.

\bibitem[Str98]{Strickland}
N.~P. Strickland, \emph{Morava {$E$}-theory of symmetric groups}, Topology
  \textbf{37} (1998), no.~4, 757--779.

\bibitem[Sul77]{Sullivan}
Dennis Sullivan, \emph{Infinitesimal computations in topology}, Inst. Hautes
  \'Etudes Sci. Publ. Math. (1977), no.~47, 269--331 (1978).

\bibitem[Tho90]{Thompson}
Robert~D. Thompson, \emph{The {$v_1$}-periodic homotopy groups of an unstable
  sphere at odd primes}, Trans. Amer. Math. Soc. \textbf{319} (1990), no.~2,
  535--559.

\bibitem[Wan14]{WangHopf}
Guozhen Wang, \emph{The monochromatic {H}opf invariant}, arXiv:1410.7292, 2014.

\bibitem[Wan15]{Wang}
\bysame, \emph{Unstable chromatic homotopy theory}, Ph.D. thesis, M.I.T., 2015.

\bibitem[Wei95]{Weiss}
Michael Weiss, \emph{Orthogonal calculus}, Trans. Amer. Math. Soc. \textbf{347}
  (1995), no.~10, 3743--3796.

\bibitem[Zhu]{Zhu2}
Yifei Zhu, \emph{Morava {E}-homology of {B}ousfield-{K}uhn functors on
  odd-dimensional spheres}, preprint.

\bibitem[Zhu15]{Zhu1}
\bysame, \emph{Modular equations for {L}ubin-{T}ate formal groups at chromatic
  level 2}, arXiv:1508.03358, 2015.

\end{thebibliography}

\end{document}